\documentclass{article}

\usepackage{arxiv}

\usepackage[utf8]{inputenc} 
\usepackage[T1]{fontenc}    
\usepackage{hyperref}       
\usepackage{url}            
\usepackage{booktabs}       
\usepackage{amsfonts}       
\usepackage{nicefrac}       
\usepackage{microtype}      
\usepackage{graphicx}
\usepackage{doi}
\usepackage{physics,amsmath}
\usepackage{authblk}
\usepackage{amsthm}
\usepackage{placeins}
\usepackage{soul}
\usepackage{xcolor}
\usepackage{float}
\usepackage{bm}
\usepackage{todonotes}
\usepackage{cite}

\usepackage{amssymb,epsfig,amscd,comment,latexsym,psfrag}
\usepackage{physics}
\usepackage{xspace,tikz}

  \usepackage{authblk}

\usepackage{subcaption}
\usepackage{array}
\usepackage{scalerel,tikz}

\usetikzlibrary{arrows}

\numberwithin{equation}{section}

\newcommand{\figurescaling}{0.585}
\def\R{\mathbb{R}}

\def\l{\left\langle}
\def\r{\right\rangle}
\def\1{\chi}

\graphicspath{ {Figures/} }

\newtheorem{theorem}{Theorem}
\newtheorem{definition}{Definition}

\newtheorem{lemma}{Lemma}
\newtheorem{example}{Example}

\theoremstyle{definition}

\title{On well-posed energy/entropy stable  boundary conditions for the  rotating shallow water equations}

\author[1]{Kenneth Duru}
\author[2,3]{Chuqiao Xu}
\affil[1]{Department of Mathematical Sciences, University of Texas at El Paso, USA}
\affil[2]{Mathematical Sciences Institute, Australian National University, Canberra, Australia}
\affil[3]{School of Mathematics, Shandong University, Jinan, China}

\date{}

\begin{document}
\maketitle

\begin{abstract}

We derive and analyze well-posed, energy- and entropy-stable boundary conditions (BCs) for the two-dimensional linear and nonlinear rotating shallow water equations (RSWE) in vector invariant form. The focus of the study is on subcritical flows, which are commonly observed in atmospheric, oceanic, and geostrophic flow applications. We consider spatial domains with smooth boundaries and formulate both linear and nonlinear BCs using mass flux, Riemann's invariants, and Bernoulli’s potential, ensuring that the resulting initial boundary value problem (IBVP) is provably entropy- and energy-stable.
The linear analysis is comprehensive, providing sufficient conditions to establish the existence, uniqueness, and energy stability of solutions to the linear IBVP. For the nonlinear IBVP, which admits more general solutions, our goal is to develop nonlinear BCs that guarantee entropy stability. We introduce the concepts of linear consistency and linear stability for nonlinear IBVPs, demonstrating that if a nonlinear IBVP is both linearly consistent and linearly stable, then, for sufficiently regular initial and boundary data over a finite time interval, a unique smooth solution exists.
Both the linear and nonlinear IBVPs can be efficiently solved using high-order accurate numerical methods. By employing high-order summation-by-parts operators to discretize spatial derivatives and implementing weak enforcement of BCs via penalty techniques, we develop provably energy- and entropy-stable numerical schemes on curvilinear meshes. Extensive numerical experiments are presented to verify the accuracy of the methods and to demonstrate the robustness of the proposed BCs and numerical schemes.
\end{abstract}


\section{Introduction}
The shallow water equations (SWE) are a fundamental set of equations in fluid dynamics, originally formulated by Saint-Venant in 1871 \cite{saintvenant1871}. They are derived by depth-averaging the Navier-Stokes equations under the assumption that the fluid layer's thickness is very small compared to the horizontal length scales of motion. The rotating shallow water equations (RSWE) extend this framework by incorporating the effects of rotation, notably through the addition of the Coriolis force term to the momentum equations, which is absent in the standard SWE. These equations are essential in various fields such as oceanography and meteorology, where they are used to model phenomena including atmospheric flows \cite{williamson1992standard,behrens1998atmospheric}, ocean currents \cite{zeitlin2018geophysical}, geophysical wave propagation, tides, and river dynamics. Due to their simplified structure, the RSWE and SWE offer a robust mathematical framework for capturing key fluid behaviors over large spatial scales, while avoiding the complexity of full three-dimensional models.

The nonlinear RSWEs are often derived in conservative form, evolving conserved variables such as mass and momentum as the prognostic quantities. However, under conditions of sufficient smoothness, the RSWEs can be reformulated into the so-called vector invariant form, which evolves primitive variables, namely mass and the velocity vector. In meteorology, for example, the vector invariant form of the RSWEs is employed to ensure exact discrete energy conservation \cite{ricardo2023conservation, thuburn2008some}, precise vorticity dynamics, and steady discrete geostrophic balance \cite{cotter2012mixed, lee2018discrete}. To reduce the influence of numerical artifacts that can contaminate simulation results, it is desirable for numerical methods to preserve important invariants inherent to the physical model. For instance, in mid-latitude weather systems, vorticity dynamics play a crucial role. Discrete conservation of vorticity helps prevent the gravitational potential—probably the largest component of atmospheric energy—from spuriously generating absolute vorticity and thereby disrupting meteorological signals \cite{staniforth2012horizontal}.

It is well-known that well-posed boundary conditions (BCs), along with their stable and accurate numerical implementations, are crucial for ensuring robust, reliable, and convergent numerical simulations of partial differential equations (PDEs) \cite{HEW2025113624,nordstrom2022linear,gustafsson1995time,kreiss1974finite}. 
Periodic BCs on cubed sphere meshes are often sufficient to accurately solve the RSWE on the sphere’s surface, thereby enabling effective global-scale atmospheric modeling \cite{thuburn2012framework,shipton2018higher,lee2018mixed,ricardo2023conservation}. However, in many practical applications—such as regional-scale or limited-area atmospheric models \cite{DaviesTerry2014, Baumhefner01011982,Caron2013,TermoniaDeckmynHamdi2009,williamson1992standard} and oceanic flow models \cite{zeitlin2018geophysical}—non-periodic, well-posed BCs with stable numerical implementations are essential for accurate and reliable simulations. For example, in ocean modeling, lateral BCs are necessary to accurately simulate Kelvin waves and associated vortical motions \cite{hsieh1983free,beletsky1997numerical}.
In other contexts, such as tsunami modeling over regional oceanic areas, the domain boundaries are not physical boundaries. Consequently, artificial, non-reflecting BCs must be employed while ensuring the well-posedness and stability of the initial boundary value problem (IBVP). In local weather prediction and regional or limited-area atmospheric models \cite{DaviesTerry2014, Baumhefner01011982,Caron2013,TermoniaDeckmynHamdi2009,williamson1992standard}, boundary data derived from global models are often prescribed at the domain boundaries. It is therefore imperative that boundary closures for regional models yield well-posed BCs and substantially reduce boundary mismatch errors \cite{DaviesTerry2014, Baumhefner01011982,Caron2013}.

The primary objective of this study is to develop well-posed and stable BCs for both the linear and nonlinear RSWE on spatial domains with smooth boundaries. A secondary goal is to formulate these BCs in a manner that facilitates their implementation using various numerical methods, such as finite difference, finite volume, finite element, and discontinuous Galerkin techniques. Additionally, the study aims to develop provably energy and entropy stable, high-order accurate numerical schemes for the linear and nonlinear IBVPs on curvilinear meshes.

The development of robust, high-order accurate numerical methods for well-posed IBVPs typically begins with establishing energy or entropy stability at the continuous level. This continuous analysis can be effectively emulated at the discrete level through the use of summation-by-parts (SBP) operators \cite{kreiss1974finite,BStrand1994,gassner2013skew,mattsson2017diagonal,williams2021provably} and careful treatment of boundary conditions, such as penalty methods like the simultaneous approximation term (SAT) \cite{carpenter1994time,Mattsson2003,HEW2025113624}.
For linear problems, the theory of IBVPs aims to identify the minimal number of boundary conditions necessary to guarantee energy stability \cite{nordstrom2022linear,GHADER20141,HEW2025113624}. In contrast, the analysis and synthesis of nonlinear hyperbolic IBVPs—without linearization—pose significant challenges. Recent efforts, however, (see, e.g., \cite{nordstrom2022nonlinear,nordstrom2024nonlinear}) have begun to address nonlinear analysis for systems such as the SWE with nonlinear boundary conditions. A key ambiguity here is that, while entropy stability—defined as the boundedness of the entropy functional—is necessary, it alone is not sufficient to establish well-posedness of nonlinear IBVPs.
Unlike linear problems, the minimal number of boundary conditions required to achieve entropy stability in nonlinear problems may differ from that prescribed by linear theory \cite{GHADER20141,gustafsson1995time}. We argue that maintaining consistency between the linear and nonlinear formulations is crucial for obtaining reliable results. Specifically, we advocate that the nonlinear IBVP, upon linearization, should produce a well-posed linear IBVP. Ensuring this linear consistency and stability facilitates the proof of the existence and uniqueness of solutions for sufficiently smooth initial and boundary data, at least over finite and sufficiently short  time intervals.

The fundamental properties of fluid dynamics modeled by the RSWE can be characterized by a dimensionless number known as the \textit{Froude number}, defined as $\textrm{Fr} = |\mathbf{u}|/\sqrt{gh}$, where $\mathbf{u}$ is the depth-averaged fluid velocity, $h$ is the water depth, and $g$ is gravitational acceleration. Physically, $h$ must be positive ($h>0$). The flow regime is classified based on the value of $\textrm{Fr}$: subcritical when $\textrm{Fr} < 1$, critical when $\textrm{Fr} = 1$, and supercritical when $\textrm{Fr} >1$. This study focuses on subcritical flows ($\textrm{Fr} <1$), which are commonly observed in atmospheric, oceanic, and geostrophic flow phenomena.

We consider spatial domains with smooth boundaries and formulate both linear and nonlinear BCs using mass flux and Bernoulli's potential, ensuring that the corresponding IBVPs are provably entropy and energy stable. Our analysis of the linear IBVP is comprehensive; similar to \cite{GHADER20141, gustafsson1995time}, it provides sufficient conditions for establishing the existence, uniqueness, and energy stability of solutions.
The nonlinear IBVP accommodates more general solutions, and our goal is to derive nonlinear BCs that guarantee entropy stability for these problems. To this end, we introduce the concepts of linear consistency and linear stability for nonlinear IBVPs. We demonstrate that if a nonlinear IBVP is both linearly consistent and linearly stable, then, for sufficiently regular initial and boundary data over finite time intervals, there exists a unique, smooth solution.
Both the linear and nonlinear IBVPs can be efficiently solved using high-order accurate numerical methods. Specifically, by employing high-order SBP operators \cite{fernandez2014review, svard2014review, BStrand1994, gassner2013skew} for spatial discretization and weak enforcement of BCs via SAT \cite{carpenter1994time, Mattsson2003, HEW2025113624}, we develop provably energy and entropy stable numerical schemes on curvilinear meshes. 
Detailed numerical experiments are presented to verify the accuracy of the methods and demonstrate the robustness of the BCs and the overall numerical framework. This work extends the results of \cite{HEW2025113624}, which addressed the 1D SWE, to the 2D linear and nonlinear RSWE on geometrically complex domains and curvilinear meshes.

The remainder of the paper is organized as follows. In Section\ref{sec:nonlinearRSWE}, we introduce the 2D RSWEs in vector invariant form and derive the evolution equation for the total energy and entropy. Section\ref{sec:linear-IBVP-RSWE} presents the linearization of the model, the derivation of boundary conditions, and the proof of well-posedness for the linear IBVP. We then develop nonlinear boundary conditions that are both entropy stable and suitable for the nonlinear RSWE, along with an analysis of the nonlinear IBVP in Section\ref{sec:nonlinear-IBVP-RSWE}. In Section\ref{sec:numerical-analysis}, we introduce high-order accurate SBP-SAT methods for both linear and nonlinear IBVPs and establish their numerical stability. Detailed numerical experiments are provided in Section\ref{sec:numerical-experiments}. Finally, Section\ref{sec:conclusion} summarizes the main findings and discusses potential directions for future research.

\section{The nonlinear RSWE in vector invariant form}\label{sec:nonlinearRSWE}
The nonlinear 2D  RSWEs in vector invariant form is given by 
\begin{equation}\label{eqn:nlSWE2D}
    \begin{cases}
    \frac{\partial h}{\partial t } + \nabla \cdot \mathbf{F} = 0, \\
    \frac{\partial \mathbf{u}}{\partial t} + \omega \mathbf{u}^\perp + \nabla G = \mathbf{0},\\
    \mathbf{F}=h\mathbf{u}, \quad G = \frac{{1 }}{{2}}|\mathbf{u}|^2 +gh, \quad \omega =  \curl{\mathbf{u}} + f_c, \quad \curl{\mathbf{u}}:= \frac{\partial v}{\partial x } - \frac{\partial u}{\partial y },\end{cases}
\end{equation}
%
where $(x,y) \in \Omega \subset \mathbb{R}^2$ are the spatial variables, $t\in[0,T]$ denotes time and $T>0$ is the final time. The prognostic flow variables are the water height $h>0$ and the  flow velocities $\mathbf{u} = [u,v]^T$, where $\mathbf{u}^\perp = [-v, u]^T$. Note that this is opposed to the flux form of the equations where the prognostics are the conserved variables, water height $h>0$ and the  momentum $\mathbf{u}h$. The diagnostic flow variables are the mass flux $ \mathbf{F}$, the Bernoulli's potential $G$ and the absolute vorticity $\omega$.
Here $f_c$ is the Coriolis frequency and $g>0$ is the constant gravitational acceleration. Furthermore, we assume that the spatial domain is sufficiently smooth and $\partial \Omega$ denotes the boundary of the domain and $\mathbf{n} = [n_x, n_y]^T \in \mathbb{R}^2$ is the outward normal unit vector on the boundary of the domain $\partial \Omega$. 

We augment the RSWEs with the smooth initial conditions
\begin{align}\label{eq:initial-condition}
    \vb{u}|_{t=0} = \vb{u}_0(x,y), \quad {h}|_{t=0} = {h}_0(x,y), \quad (x,y) \in \Omega.
\end{align}
We also need appropriate BCs at the boundary of the domain $\partial \Omega$ in order to close the system and ensure a well-posed IBVP. This will be discussed in detail later in the next sections.

In meteorology, for example, the vector-invariant form of the SWEs \eqref{eqn:nlSWE2D} is often employed to achieve exact discrete energy conservation \cite{ricardo2023conservation, lee2018discrete, lee2018mixed}, precise vorticity dynamics, and discrete steady geostrophic balance. In the present work, we focus on the nonlinear energy balance for regional models, particularly in scenarios involving non-periodic boundary conditions.

We define the elemental energy $e$, as the sum of the kinetic energy and potential energy, and the total energy $E(t)$ as the integral of the elemental energy over the spatial domain
\begin{equation}\label{eq:elemental}
   e=\tfrac{1}{2}h|\vb{u}|^2 + \tfrac{1}{2}gh^2, \quad E(t) = \int_{\Omega} e d\Omega.
\end{equation}
For subcritical flows, with $\mathrm{Fr} = |\vb{u}|/\sqrt{gh} < 1$, the elemental energy is a convex function of the prognostic variables ($h,\vb{u}$) and thus defines a mathematical entropy \cite{ricardo2023conservation,HEW2025113624}. 
\begin{definition}\label{def:stability}
    The RSWE \eqref{eqn:nlSWE2D} with the initial condition \eqref{eq:initial-condition}, appropriate BCs and homogeneous boundary data is called stable  if $dE(t)/dt \le 0$ for all $t \in [0, T]$.
\end{definition}
It is therefore desirable for numerical methods to be designed to bound the total energy and entropy, thereby ensuring nonlinear stability and robustness. The elemental energy $e$ satisfies a continuity equation. To demonstrate this, we take the time derivative of the elemental energy and utilize \eqref{eqn:nlSWE2D} to eliminate the time derivatives of the prognostic variables, yielding
%
\begin{equation}
     \frac{\partial e}{\partial t } = \vb{F}\cdot \frac{\partial \vb{u}}{\partial t } + G\frac{\partial h}{\partial t }   = -\underbrace{\vb{F}\cdot (\omega \vb{u}^\perp)}_{h\omega (\vb{u}\cdot\vb{u}^\perp)=  h\omega(u v - uv) =0} - \vb{F}\cdot \nabla G - G\nabla \cdot \vb{F}.
 \end{equation}
 We have the continuity equation for the energy $e$
\begin{equation}\label{eq:energy_continuity}
    \frac{\partial e}{\partial t } + \nabla \cdot \big(G\vb{F}\big) = 0, \quad G\vb{F} = \frac{1}{2}gh\vb{F} + e \vb{u},
\end{equation}
where we have used the identities 
$$
\vb{F}\cdot  \vb{u}^\perp =  h (u v - uv) = 0, \quad \nabla \cdot \big(G\vb{F}\big) = \vb{F}\cdot \nabla G + G\nabla \cdot \vb{F}.
$$
Note that  $G\vb{F} = (\frac{1}{2}gh\vb{F} + e \vb{u})$ is the energy flux. Thus the energy flux decomposes into a pressure flux $\frac{1}{2}gh\vb{F}$ due to gravity $g>0$ and a transport flux $e \vb{u}$ which is transported by the flow $\vb{u}$.

Integrating equation \eqref{eq:energy_continuity} over the spatial domain yields the conservation of total energy,
\begin{align}\label{eq:energy_conservation}
 \frac{d}{dt}E(t)= \frac{d}{dt}\int_{\Omega} e d\Omega =\mathrm{BT}:= -\oint_{\partial \Omega} G F_n dS, \quad G F_n = \frac{1}{2}gh^2u_n + eu_n,
\end{align}
where $\mathrm{BT}$ is the boundary term, $F_n = \vb{n}\cdot \vb{F} = hu_n$   is the normal mass flux, with $u_n=\vb{n}\cdot \vb{u}$ being the normal velocity on the boundary $\partial \Omega$, and $G$ is the Bernoulli's potential. On a periodic domain the contour integral on the right hand side of \eqref{eq:energy_conservation} vanishes, granting the conservation of total energy, that is $E(t)=E(0)$ for all $t \in [0,T]$, and thus ensuring stability. Periodic BCs on smooth geometries  are often sufficient for the well-posedness of global models posed on the surface of the sphere. At the discrete level, periodic BCs can be implemented by designing appropriate numerical fluxes or SATs that cancel the boundary term $\mathrm{BT}$ or ensure $\mathrm{BT} \le 0$ on computational domain boundaries. 

For regional models defined on bounded domains, non-periodic BCs are essential to accurately represent various physical phenomena at the boundaries \cite{hsieh1983free, beletsky1997numerical}. They enable the enforcement of global data constraints on domain boundaries, for example, \cite{DaviesTerry2014, Baumhefner01011982, Caron2013, TermoniaDeckmynHamdi2009, Williamson1992}. Consequently, well-posed BCs are crucial for ensuring model stability and the convergence of numerical solutions. 
While the theory of well-posedness for linear hyperbolic IBVPs is relatively well-developed, the corresponding theory for nonlinear hyperbolic IBVPs remains less complete. A primary objective of this study is to develop well-posed and stable nonlinear BCs for the nonlinear RSWEs \eqref{eqn:nlSWE2D} on spatial domains $\Omega$ with smooth boundaries $\partial \Omega$. 
A second goal is to formulate these nonlinear BCs in a manner that facilitates their implementation within numerical methods. Lastly, the study aims to develop provably energy- and entropy-stable numerical schemes for the nonlinear IBVP on smooth geometries.

For linear problems, the theory of IBVPs aims to determine the minimal number of BCs necessary to establish energy stability \cite{HEW2025113624,GHADER20141,gustafsson1995time,nordstrom2022linear}. In contrast, for nonlinear problems, the minimal number of BCs required to prove stability may differ from that prescribed by linear theory \cite{nordstrom2022nonlinear,nordstrom2024nonlinear}. We contend, however, that maintaining consistency between the linear and nonlinear IBVPs is crucial for obtaining reliable results. Specifically, we argue that the nonlinear IBVP, when linearized, should produce a well-posed linear IBVP. This linear consistency is essential for establishing the existence of unique solutions for sufficiently smooth initial and boundary data, at finite times.

\section{Linear theory of IBVP for the RSWE}\label{sec:linear-IBVP-RSWE}
In this section we will linearize the nonlinear RSWE  \eqref{eqn:nlSWE2D}, and  give a quick introduction to the theory of IBVP for the linear hyperbolic PDE. We will apply the theory to the linearized RSWE and derive well-posed BCs.
To begin, we introduce zero-mean quantities in the form of perturbed variables, i.e., $u = U + \widetilde{u}$, $v = V + \widetilde{v}$  and $h = H +\widetilde{h}$ with the constant mean states $H > 0$, $U$, $V$. Discarding nonlinear terms of order $\mathcal{O}(\widetilde{u}^2, \widetilde{v}^2, \widetilde{h}^2,\widetilde{u}\widetilde{h}, \widetilde{v}\widetilde{h}, \widetilde{u}\widetilde{v})$, we obtain 
the 2D linear RSWE in vector invariant form 
\begin{equation}\label{eqn:LinSWE2D}
    \begin{cases}
    \frac{\partial h}{\partial t } + \nabla \cdot \mathbf{F} = 0, \\
    \frac{\partial \mathbf{u}}{\partial t} + \omega \mathbf{U}^\perp + f_c \mathbf{u}^\perp + \nabla G = \mathbf{0},\\
     \mathbf{F}= H\mathbf{u} + \mathbf{U}h, \quad G = \mathbf{U}\cdot\mathbf{u} +gh, \quad\omega =  \curl{\mathbf{u}}, \quad \curl{\mathbf{u}}:= \frac{\partial v}{\partial x } - \frac{\partial u}{\partial y },  
    \end{cases}
\end{equation}
where $\mathbf{U} = [U, V]^T$ and $\mathbf{U}^\perp = [-V, U]^T$, and we have dropped the tilde on the fluctuating fields for convenience.
Note that $\mathbf{F} = Uh + Hu$ is the linear mass flux and $G = \mathbf{U}\cdot\mathbf{u} +gh$ is the linear Bernoulli's potential.
As before,  for the linear RSWE \eqref{eqn:LinSWE2D} we define the elemental energy $e$ as the sum of the kinetic energy and potential energy, and the total energy $E(t)$ as the integral of the elemental energy over the spatial domain
\begin{equation}\label{eq:energy_linear}
   e=\tfrac{1}{2}H|\vb{u}|^2 + \tfrac{1}{2}gh^2, \quad  E(t) = \int_{\Omega} e d\Omega,
\end{equation}
where $g>0$ and $H>0$. Note that the total energy is a weighted $L_2$-norm of the prognostic variables $(h, \vb{u})$. Thus a bound on the energy will establish the stability and, hopefully, the well-posedness of the linear RSWE \eqref{eqn:LinSWE2D}.
As in the nonlinear case we show that elemental energy evolves according to a conservation law.
The continuity equation for the energy $e$ defined in \eqref{eq:energy_linear} is given by
\begin{equation}\label{eq:energy_continuity_linear}
    \frac{\partial e}{\partial t } + \nabla \cdot \big(gHh\vb{u} + e\vb{U}\big) = 0.
\end{equation}
Similar to the nonlinear RSWE  as above, note that  the energy flux $gHh\vb{u} + e\vb{U}$ decomposes into a pressure flux $gHh\vb{u}$ due to gravity $g>0$ and a transport flux $e\vb{U}$ which is transported by the background flow $\vb{U}$.

Integrating equation \eqref{eq:energy_continuity_linear} over the spatial domain yields the conservation of total energy,
\begin{align}\label{eq:energy_conservation_linear}
 \frac{d}{dt}E(t)= \frac{d}{dt}\int_{\Omega} e d\Omega = \mathrm{BT}, \quad \mathrm{BT}:=-\oint_{\partial \Omega} \left(gHh{u}_n + e{U}_n\right) dS, 
\end{align}
where $\mathrm{BT}$ is the boundary term, $u_n = \vb{n}\cdot \vb{u}$ and $U_n = \vb{n}\cdot \vb{U}$ are  the normal velocities on the boundaries. 
On a periodic domain the boundary term $\mathrm{BT}$, defined by the contour integral on the right hand side of \eqref{eq:energy_conservation_linear} vanishes, granting the conservation of total energy, that is $E(t)=E(0)$ for all $t \in [0,T]$, and thus ensuring stability of periodic solutions. At the discrete level, periodic BCs can be implemented by designing appropriate numerical fluxes or SATs that cancel the contour integral of the energy flux on computational domain boundaries. 

For regional models with non-periodic BCs, stability alone does not guarantee well-posedness of the IBVP. The way in which BCs are defined can either eliminate the existence of solutions or result in a non-unique solution space, thereby rendering the IBVP ill-posed. Fortunately, the theory of well-posedness for linear hyperbolic IBVPs is well-developed. We will provide a brief overview and refer the reader to \cite{GHADER20141, gustafsson1995time, nordstrom2022linear} for a more comprehensive discussion.

\subsection{Well-posedness of the linear RSWE IBVP}
We will begin by rewriting the linear RSWE \eqref{eqn:LinSWE2D} in standard form, as a system of first order hyperbolic PDE. We define the unknown vector field $\mathbf{q} = [h, u, v]^T$.
The IBVP for the linear RSWE in two space dimensions  can be formulated as follows:
\begin{subequations}
\label{eq6-t}
\renewcommand{\theequation}{\theparentequation\alph{equation}}
\begin{equation}
\label{eq6a-t}
\frac{\partial \mathbf{q}}{\partial t } = D\mathbf{q}, \quad D\mathbf{q}:= -\left(
  A\frac{\partial \mathbf{q}}{\partial x } + B \frac{\partial \mathbf{q}}{\partial y } + C\mathbf{q} \right), \quad (x,y) \in \Omega, \quad t > 0,
\end{equation}
\begin{equation}
\label{eq6b-t}
\textbf{q}(x,y,t) = \mathbf{q}_0(x,y),\quad  (x,y) \in \Omega,\quad  t = 0,
\end{equation}
\begin{equation}
\label{eq6c-t}
\mathcal{B}\mathbf{q}(x,y,t) = \mathbf{d}(x,y,t),\quad  (x,y) \in \partial \Omega ,\quad  t \geq 0,
\end{equation}
\end{subequations}
where the coefficients matrices are given by
\begin{equation}
    \label{eq-m-1}
A = \begin{bmatrix}
U & H & 0\\
g & U & 0\\
0 & 0 & U
\end{bmatrix}, \quad
B = \begin{bmatrix}
V & 0 & H\\
0 & V & 0\\
g & 0 & V
\end{bmatrix},\quad
C = \begin{bmatrix}
0 & 0 & 0\\
0 & 0 & -f_c\\
0 & f_c & 0
\end{bmatrix}.
\end{equation}
Here $\mathbf{q}_0$ and $\mathbf{d}$ are the initial and boundary data, $\mathcal{B}$ is  the boundary operator which enforces BCs on $\mathbf{q}$ at the boundary $\partial \Omega$ of the domain $\Omega$. Note that the boundary operator $\mathcal{B}$ is not defined yet, and needs to be determined so that the IBVP \eqref{eq6a-t}--\eqref{eq6c-t} is well-posed.  In the analysis below, we will consider homogeneous boundary data, however, the analysis can be extended to non-homogeneous boundary data, but this would complicate the algebra.

In the following, we will introduce the relevant notation required to derive the boundary operator $\mathcal{B}$ and prove the well-posedness of the IBVP \eqref{eq6a-t}--\eqref{eq6c-t}. 
To begin, we define the weighted $L^2(\Omega)$ inner product and norm
\begin{equation}\label{eqn:L2norm}
 (\mathbf{p},\mathbf{q})_W := \int_{\Omega } \mathbf{p}^TW\mathbf{q} \; d\Omega, \quad  \|\mathbf{q}\|^2_W := (\mathbf{q},\mathbf{q})_W = \int_{\Omega } e \; d\Omega,   \quad 
 W =  \frac{1}{2}\begin{bmatrix}
g & 0 & 0\\
0 & H & 0\\
0 & 0 & H
\end{bmatrix}.
\end{equation}
Note that  $W$ is diagonal and positive, and $e=\mathbf{q}^TW\mathbf{q} >0, \quad \forall \mathbf{q} = [h,u,v]^T  \in \mathbb{R}^3 \backslash \{\mathbf{0}\}$, and $E=\|\mathbf{q}\|^2_W>0$, where $e$ is the elemental energy and $E$ is the total energy defined in \eqref{eq:energy_linear}. It will be useful to introduce the weighted matrices
\begin{align}
     \label{eqn:weighted-matrices}
WA = \frac{1}{2}\begin{bmatrix}
gU & gH & 0\\
Hg & HU & 0\\
0 & 0 & HU
\end{bmatrix}, \quad
WB = \frac{1}{2}\begin{bmatrix}
gV & 0 & gH\\
0 & HV & 0\\
Hg & 0 & HV
\end{bmatrix},
\quad
WC = \frac{H}{2}\begin{bmatrix}
0 & 0 & 0\\
0 & 0 & -f_c\\
0 & f_c & 0
\end{bmatrix}.
\end{align}
Note that the weighted matrices $WA$ and $WB$ are symmetric  and the weighted matrix $WC$ is an anti-symmetric matrix, that is $(WA)^\top=WA$, $(WB)^\top=WB$, and $(WC)^\top = - WC$.
\begin{definition}\label{eqn:wellposed}
The IBVP \eqref{eq6a-t}--\eqref{eq6c-t} is well-posed if there exists a unique solution, $\mathbf{q} = [h, u, v]^T$, which satisfies 
$$\| \mathbf{q}\|_W \leq K e^{\mu t}\| \mathbf{q}_0\|_W ,$$
for  some constants $K>0$, $\mu$ $\in \mathbb{R}$ independent of the initial data $\mathbf{q}_0$.
\end{definition}
The well-posedness of the IBVP \eqref{eq6a-t}--\eqref{eq6c-t} can be related to the boundedness of the differential operator $D$. We introduce the function space 
\begin{equation}\label{eq:functionspace_V}
\mathbb{V}=\left\{\mathbf{q} \mid \, \mathbf{q}(x,y) \in \mathbb{R}^3, \quad\|\mathbf{q}\|_W<\infty, \quad x \in \Omega, \quad\{\mathcal{B} \mathbf{q}=0, \quad (x,y) \in \partial \Omega\}\right\}.
\end{equation}
The following definition will be useful.
\begin{definition}\label{def:semibounded}
The differential operator $D$ is semi-bounded in the function space $\mathbb{V}$ if $\exists \, \mu \in \mathbb{R}$ independent of $\mathbf{q} \in \mathbb{V}$ such that
$$ 
(\mathbf{q}, D\mathbf{q})_W \leq \mu \|\mathbf{q}\|^2_W, \; \forall \mathbf{q} \in \mathbb{V}.
$$
\end{definition}
\begin{lemma}\label{lem:semi-bounded_BT}
Consider the linear differential operator { $D$} given in \eqref{eq6a-t} subject to the BCs \eqref{eq6c-t}, $\mathcal{B}\mathbf{q} = 0$.
Let $\mathrm{BT} = -\oint_{\partial \Omega}(gHh{u}_n + e{U}_n) dS$ be the boundary term given in \eqref{eq:energy_conservation_linear} and $W$ be the diagonal and positive definite  weight matrix defining the weighted $L^2$-norm \eqref{eqn:L2norm}. If $\mathcal{B}\mathbf{q} = 0$ is such that  the boundary term   $\mathrm{BT}\le 0$, then $D$ is semi-bounded.
\end{lemma}
\begin{proof}
Now consider $(\mathbf{q}, D\mathbf{q})_{W  }$ and recall that the weighted constant coefficients matrices given by \eqref{eqn:weighted-matrices} are symmetric. We integrate by parts and we have 
$$
(\mathbf{q}, D\mathbf{q})_{W  } + (\mathbf{q}, D\mathbf{q})_{W  }  = \mathrm{BT}.
$$
So for the boundary operator $\mathcal{B}\mathbf{q} = 0$, if $\mathrm{BT}\le 0$ then 
$$(\mathbf{q}, D\mathbf{q})_{W  } = \frac{1}{2}\mathrm{BT} \le 0.$$
 An upper bound is the case of $\mu=0$. Clearly for $(\mathbf{q}, D\mathbf{q})_{W  } =0$,  $D$ is semi-bounded.
\end{proof}
\noindent
It is often possible to formulate BCs $\mathcal{B}\mathbf{q} =0$ such that $\operatorname{BT} \le 0$. An immediate example is the case of periodic BC, where $\operatorname{BT} = 0$. However, for non-periodic BCs, it is imperative that the boundary operator must not destroy the  existence and uniqueness of solutions. We will now introduce the definition of maximally semi-boundedness of $D$ which will ensure well-posedness of the IBVP \eqref{eq6a-t}--\eqref{eq6c-t}.
\begin{definition}\label{eqn:maximal}
The differential operator $D$ defined in \eqref{eq6a-t} is maximally semi-bounded if it is semi-bounded in the function space $\mathbb{V}$
but not semi-bounded in any function space with fewer BCs.
\end{definition}
\noindent
The maximally semi-boundedness property  is intrinsically connected to well-posedness of the IBVP. We will formulate this result in the following theorem. The reader can consult \cite{gustafsson1995time} for more elaborate discussions.

\begin{theorem}\label{thm:swp}
    Consider the IBVP \eqref{eq6a-t}--\eqref{eq6c-t} if the differential operator $D$ is maximally semi-bounded, $(\mathbf{q}, D\mathbf{q})_W \leq \mu \|\mathbf{q}\|^2_W$, then it is  well-posed. That is, there is a unique solution $\mathbf{q}$ satisfying the estimate
    $$\| \mathbf{q}\|_W \leq K e^{\mu t}\| \mathbf{q}_0\|_W , \quad {K =1}.$$
\end{theorem}
\begin{proof}
    We consider 
\begin{align}
\begin{split}  
     \frac{d}{dt} \| \mathbf{q}\|_W^2   = \left(\mathbf{q},\frac{\partial\mathbf{q}}{\partial t}\right)_W  + \left(\frac{\partial\mathbf{q}}{\partial t}, \mathbf{q}\right)_W =\left(\mathbf{q},D\mathbf{q}\right)_W + \left(D\mathbf{q}, \mathbf{q}\right)_W  
\end{split}
    \end{align}
    Semi-boundedness  yields
\begin{align}
\begin{split}  
     \frac{d}{dt} \| \mathbf{q}\|_W^2   \leq 2\mu \| \mathbf{q} \|^2_W  \iff
    \frac{d}{dt} \| \mathbf{q} \|_W \leq  \mu \| \mathbf{q}\|_W . \\
\end{split}
    \end{align}
    Gr\"onwall's Lemma  gives
    \begin{align}
    \|\mathbf{q}\|_W \leq e^{\mu t}\|\mathbf{q}_0\|_W , 
\end{align}
With $K=1$ we have the required result of well-posedness given by Definition \ref{eqn:wellposed}.
\end{proof}
\noindent
Note that by Lemma \ref{lem:semi-bounded_BT}, $\mathrm{BT} \le 0$ implies that the differential operator $D$ is semi-bounded in the function space $\mathbb{V}$. Thus to ensure maximally semi-boundedness we will need to determine the minimal number of BCs such that $\mathrm{BT} \le 0$.
It is also noteworthy that while we have considered homogeneous boundary data here the analysis can be extended to non-homogeneous boundary data, in particular when $\mu < 0$. For more elaborate discussions and  examples see  \cite{nordstrom2022linear,GHADER20141,NORDSTROM2020112857} and the references therein. Furthermore, numerical experiments performed later in this study confirms that our results extend to non-homogeneous boundary data. 
%
\subsection{Well-posed linear BCs}
    We will now formulate well-posed BCs for the linear IBVP \eqref{eq6a-t}--\eqref{eq6c-t}.
Well-posed BCs require that the  differential operator $D$ to be  maximally semi-bounded in the function space $\mathbb{V}$. That is, we need a minimal number of BCs so that $D$ is semi-bounded in $\mathbb{V}$.  First,  we will determine the minimal number of BCs needed such that $\mathrm{BT} \le 0$  and proceed later  to give the forms of the BCs.

\paragraph{The number of  BCs for subcritical flows.}
To begin, we consider the outward normal unit vector $\mathbf{n} = [n_x, n_y]^T$ and  introduce the variables
\begin{align}
U_n = n_xU + n_y V, \quad u_n = n_xu + n_y v, \quad \quad u_s = n_yu - n_x v, \quad c = \sqrt{gH}>0,
\end{align}
where $U_n$, $u_n$ are the normal velocities on the boundary, $u_s$ is the tangential velocity on the boundary and $c>0$ is the speed of gravity waves. The boundary is called an \textbf{inflow boundary} when $U_n <0$ and an  \textbf{outflow boundary} when $U_n \ge 0$. We also introduce the dimensionless variables $\mathbf{p}$ and the boundary matrix $M$ defined by
\begin{align}
\mathbf{p}=
\begin{bmatrix}
    h^{\prime}\\
    u_n^{\prime}\\
    u_s^{\prime}
\end{bmatrix}
=
\begin{bmatrix}
    h/H\\
    u_n/c\\
    u_s/c
\end{bmatrix},
\quad 
M = \begin{bmatrix}
    U_n & c & 0\\
    c & U_n & 0\\
    0 & 0 & U_n
\end{bmatrix}.
\end{align}
Consider the boundary term
$$
\mathrm{BT} = -\oint_{\partial \Omega} \left(gHh{u}_n + e{U}_n\right) dS= -\frac{c^2H}{2}\oint_{\partial \Omega}{\left(
\mathbf{p}^TM\mathbf{p}\right)} dS.
$$
By using the eigen-decomposition ${M} = S\Lambda S^{T}$ given by
\begin{eqnarray}
\label{eq:eigen_decompose}
S = \frac{1}{\sqrt{2}}\begin{bmatrix}
        1&1 & 0\\
        1&-1&0\\
        0 & 0 & \sqrt{2}
    \end{bmatrix},
    \quad
    \Lambda = \begin{bmatrix}
        \lambda_1&0&0\\
        0&\lambda_2&0\\
        0&0&\lambda_3\\
    \end{bmatrix},
    \quad 
    \lambda_1 = U_n+c, \quad  \lambda_2 = U_n-c, \quad \lambda_3 = U_n,
\end{eqnarray}
with the linear transformation
\begin{eqnarray}
\label{eq:linear_transform}
    &\begin{bmatrix}
            w_1\\
            w_2\\
            w_3
        \end{bmatrix} 
        = {S}^{\top} \textbf{p} =  \frac{1}{\sqrt{2}}
    \begin{bmatrix}
        h^{\prime}+u_n^{\prime} \\ 
        h^{\prime}-u_n^{\prime}\\
        \sqrt{2}u_s^{\prime}
    \end{bmatrix},
\end{eqnarray}
the boundary term $\mathrm{BT}$ can be re-written as 
  \begin{eqnarray}
        \label{eq:boundary_term_0}
\mathrm{BT}=-\frac{c^2H}{2}\oint_{\partial \Omega}
\left(\lambda_1 w_1^2 + \lambda_2 w_2^2 + \lambda_3 w_3^2 \right)dS.
\end{eqnarray}  
The number of BCs will depend on the signs of the eigenvalues $\lambda_1$, $\lambda_2$, $\lambda_3$, which in turn depend on the magnitude of the normal flow velocity $U_n$ relative to the characteristic wave speed $c$, and determined by the Froude number $\operatorname{Fr} = |\vb{U}|/c$.  In particular, the number of BCs must be equal to the number of negative eigenvalues, $\lambda_1$, $\lambda_2$, $\lambda_3$, of the boundary matrix $M$.  
For subcritical flows with $0\le \operatorname{Fr} <1$, then $\lambda_1>0$, $\lambda_2 <0$, and $\lambda_3 = U_n$ takes the sign of the normal background  flow velocity $U_n$. Thus at the inflow, we have $\lambda_3 = U_n <0$ and at the outflow we have $\lambda_3 = U_n \ge 0$. The number of BCs  are summarized in the Table \ref{tab:num_bound} below for different flow conditions.
\begin{table}[htb!]
\centering
\begin{tabular}{ | m{7.75em}| m{5em} | m{5em}| m{4em}| m{7em}|} 
  \hline
                   Type of boundary                 & $\lambda_1=U_n+c$ & $\lambda_2=U_n-c$ & $\lambda_3=U_n$ &Number of BCs\\ 
  \hline
   Inflow: $U_n <0$ & $>0$ & $<0$ & $<0$ & 2 \\ 
  \hline
   Outflow: $U_n \ge 0$ & $>0$ & $<0$ & $\ge 0$ & 1 \\  
  \hline
\end{tabular}
\caption{The signs of the eigenvalues and number of the BCs for subcritical flows.}
\label{tab:num_bound}
\end{table}

\paragraph{Well-posed and stable BCs for subcritical flows.} For subcritical flows, with $0\le \mathrm{Fr}< 1$, we formulate the inflow BCs, when $U_n \ne 0$, 
    \begin{eqnarray}\label{eq:BC-sub-critical_inflow}
    \{\mathcal{B}\textbf{p}=\textbf{d},  \ (x,y) \in \partial \Omega \} \equiv \{ {w_2 - \gamma w_1 = d_1}; \ {w_3 = d_2}; \ \text{if} \ U_n < 0\},
    \end{eqnarray}
and the outflow BC
    \begin{eqnarray}\label{eq:BC-sub-critical_outflow}
    \{\mathcal{B}\textbf{p}=\textbf{d},  \ (x,y) \in \partial \Omega \} \equiv \{ {w_2 - \gamma w_1 = d_1};  \ \text{if} \ U_n \ge 0\},
    \end{eqnarray}
    where $d_1$ and $d_2$ are boundary data.
Here $\gamma \in \mathbb{R}$ is a boundary reflection coefficient. The following Lemma constraints the boundary reflection coefficient $\gamma$.
\begin{lemma}\label{Lem:BC-Sub-critical flow}
    Consider the boundary term $\operatorname{BT}$ defined in \eqref{eq:boundary_term_0} and the BC \eqref{eq:BC-sub-critical_inflow}--\eqref{eq:BC-sub-critical_outflow} with homogeneous boundary data $\mathbf{d} =0$ for sub-critical flows $ 0\le \operatorname{Fr} < 1$ with $\lambda_1 >0$ and $\lambda_2 <0$. If  $0\le \gamma^2 \le -\lambda_1/\lambda_2$, then the boundary term is never positive, that is $\mathrm{BT} \le 0$.
\end{lemma}
\begin{proof}
   Let  $w_2 = \gamma w_1$ and consider
        $ \lambda_1 w_1^2 + \lambda_2 w_2^2     =  w_1^2 \left( \lambda_1  + \lambda_2 \gamma^2\right).$
    Note that $\lambda_1>0$, $\lambda_2<0$, and if  $ \gamma^2 \le -\lambda_1/\lambda_2$ then  $\left( \lambda_1  + \lambda_2 \gamma^2\right)\ge 0$. Thus when $\lambda_3 = U_n <0$ the inflow BC \eqref{eq:BC-sub-critical_inflow} gives
    $$
    \mathrm{BT}=-\frac{c^2H}{2}\oint_{\partial \Omega}
\left(\lambda_1 w_1^2 + \lambda_2 w_2^2 + \lambda_3 w_3^2 \right)dS =  -\frac{c^2H}{2}\oint_{\partial \Omega}
\left(\left( \lambda_1  + \lambda_2 \gamma^2\right) w_1^2\right)dS \le 0.
    $$
    Similarly, if  $ \gamma^2 \le -\lambda_1/\lambda_2$ then  $\left( \lambda_1  + \lambda_2 \gamma^2\right)\ge 0$ and  $\lambda_3 = U_n \ge 0$, the outflow BC \eqref{eq:BC-sub-critical_outflow} gives
     $$
    \mathrm{BT}=-\frac{c^2H}{2}\oint_{\partial \Omega}
\left(\lambda_1 w_1^2 + \lambda_2 w_2^2 + \lambda_3 w_3^2 \right)dS =  -\frac{c^2H}{2}\oint_{\partial \Omega}
\left( \left( \lambda_1  + \lambda_2 \gamma^2\right) w_1^2 + \lambda_3 w_3^2\right)dS \le 0.$$
\end{proof}
\noindent
We will summarize this section with the theorem which proves the well-posedness of the IBVP defined by the vector invariant form of the linear SWE \eqref{eqn:LinSWE2D} with the initial condition \eqref{eq:initial-condition}, and the BCs  $\mathcal{B}\mathbf{q}=0$, where $\mathcal{B}\mathbf{q} $ is given by \eqref{eq:BC-sub-critical_inflow}--\eqref{eq:BC-sub-critical_outflow}.
\begin{theorem}\label{theo:well-posdeness-ibvp}
    Consider the IBVP defined by the vector invariant form of the linear RSWE \eqref{eqn:LinSWE2D} with the initial condition \eqref{eq:initial-condition}, and the BC  $\mathcal{B}\mathbf{q}=0$, where $\mathcal{B}\mathbf{q} $ is given by \eqref{eq:BC-sub-critical_inflow}--\eqref{eq:BC-sub-critical_outflow}
     with $0\le \gamma^2 \le -\lambda_1/\lambda_2$.  For subcritical flows, with $0\le\mathrm{Fr} = |\mathbf{U}|/\sqrt{gH} < 1$, the IBVP is well-posed.
     That is, there is a unique $\mathbf{q} = [h,u,v]^T$ that satisfies the estimate,
$$
\|\mathbf{q}\|_W \leq \|\mathbf{q}_0\|_W, \quad \forall \, t \in [0, T],
$$
     where $\mathbf{q}_0 = [h_0,u_0,v_0]^T$ is the compactly supported initial data.
\end{theorem}
\begin{proof}
Invoking Lemma \ref{lem:semi-bounded_BT}, Theorem \ref{thm:swp}, and Lemma \ref{Lem:BC-Sub-critical flow} completes the proof of the theorem.
    \end{proof}
Theorem \ref{theo:well-posdeness-ibvp} establishes the existence and stability of a unique solution for the linear IBVP. However, the BCs $\mathcal{B}\mathbf{q}=0$  given by \eqref{eq:BC-sub-critical_inflow}--\eqref{eq:BC-sub-critical_outflow} are a bit cryptic. We will give some physically relevant BCs which are important in several modeling scenarios. We will test the physical BCs against Theorem \ref{theo:well-posdeness-ibvp} to determine if they  give well-posed IBVPs when coupled to the linear RSWE \eqref{eqn:LinSWE2D}.

{\example[Linear Riemann invariants]{
Riemann invariants serve as natural carriers of information in hyperbolic PDE systems and are essential for designing effective boundary conditions. For example, they facilitate the transport of quantities such as mass, pressure, or energy from the boundaries into the computational domain. Additionally, Riemann invariants can be employed to derive transparent or absorbing boundary conditions, which are crucial for minimizing unwanted reflections at the artificial boundaries of a computational domain in regional models and open systems.

We will distinguish an inflow boundary with $U_n < 0$ and an outflow boundary with $U_n\ge 0$. Here the Riemann invariants on the boundary are
\begin{align}
  r_1:= \sqrt{\frac{g}{H}}h + u_n, \quad r_2 := \sqrt{\frac{g}{H}}h - u_n,  \quad r_3 := u_s.   
\end{align}
Note that at an inflow boundary with $U_n<0$, $r_1$ with the characteristic speed $\lambda_1 = U_n + c >0$ is the outgoing Riemann invariant, while $r_2$ and $r_3$  with the characteristic speeds $\lambda_2 = U_n - c < 0$, $\lambda_3 = U_n<0$ respectively, are the incoming Riemann invariants on the boundary. However, at an outflow boundary with $U_n \ge 0$, $r_1$ and $r_3$ are the outgoing  Riemann invariants and $r_2$ is the incoming Riemann invariant on the boundaries. We will  specify BCs by sending boundary data $\mathbf{d}$ through the incoming Riemann invariants. 

For $U_n< 0$ we formulate the inflow BC
\begin{eqnarray}\label{eq:BC-sub-critical_Riemann_inflow}
\mathcal{B}\textbf{q}=
 \begin{cases}
    r_2 := \sqrt{\frac{g}{H}}h - u_n= d_1 \equiv w_2 - \gamma w_1 = d_1/c, \quad \gamma = 0,\\
   r_3 := u_s = d_2 \equiv  \ {w_3 = d_2/c},
   \\
    \end{cases}
    \ \text{if} \ U_n < 0.
    \end{eqnarray}
    Note that if $U_n < 0$, then $\gamma^2 =0 \le -\lambda_1/\lambda_2 < 1.$
    When $U_n \ge 0$ we have
    \begin{eqnarray}\label{eq:BC-sub-critical_Riemann_outflow}
\mathcal{B}\textbf{q}=
    r_n := \sqrt{\frac{g}{H}}h - u_n= d_1 \equiv w_2 - \gamma w_1 = d_1/c, \quad \gamma = 0,
    \ \text{if} \ U_n \ge 0.
    \end{eqnarray}
    Similarly, if $U_n \ge 0$, then $\gamma^2 =0 < 1 \le -\lambda_1/\lambda_2.$
    Therefore the inflow and outflow BCs \eqref{eq:BC-sub-critical_Riemann_inflow}--\eqref{eq:BC-sub-critical_Riemann_outflow} satisfy Lemma \ref{Lem:BC-Sub-critical flow},  and will give a well-posed IBVP in the sense of Theorem \ref{theo:well-posdeness-ibvp}.  Note that with homogeneous data $\mathbf{d}=0$ we have the so-called absorbing/transmissive BCs.
}}

{\example[Linear mass flux]{
The mass flux can be used to control how mass and materials are transported through the boundaries into the computational domain. For example the no mass flux BC can used to ensure that no mass is transported across the boundary. At the inflow, with $U_n<0$, we formulate the inflow  mass flux BC
\begin{eqnarray}\label{eq:BC-sub-critical_massflux_inflow}
\mathcal{B}\textbf{q}=
 \begin{cases}
    F_n := U_nh + Hu_n= d_1 \equiv w_2 - \gamma w_1 = d_1/H, \quad \gamma = -\lambda_1/\lambda_2 = -\frac{U_n + c}{U_n-c},\\
   F_s := u_s = d_2 \equiv  \ {w_3 = {\frac{d_2}{c}}},
    \end{cases}
    \ \text{if} \ U_n < 0.
    \end{eqnarray}
    As before, note that if $U_n < 0$, then $\gamma^2 =\lambda_1^2/\lambda_2^2 < -\lambda_1/\lambda_2 < 1$. 
    Therefore the inflow BCs \eqref{eq:BC-sub-critical_massflux_inflow} satisfy Lemma \ref{Lem:BC-Sub-critical flow},  and will result to a well-posed IBVP in the sense of Theorem \ref{theo:well-posdeness-ibvp}, when coupled to linear RSWE in vector in variant form \eqref{eqn:LinSWE2D}.
    
    At the outflow, with $U_n\ge 0$, the outflow mass flux BC is
    \begin{eqnarray}\label{eq:BC-sub-critical_massflux_outflow}
\mathcal{B}\textbf{q}=
    F_n := U_nh + Hu_n= d_1 \equiv w_2 - \gamma w_1 = d_1/H, \quad \gamma = -\lambda_1/\lambda_2 = -\frac{U_n + c}{U_n-c},
    \quad \text{if} \ U_n \ge  0.
    \end{eqnarray}
    Note that  if $U_n \ge 0$, then $\gamma^2 = \lambda_1^2/\lambda_2^2 \ge   -\lambda_1/\lambda_2 \ge 1.$
    In particular, when $U_n = 0$, then we have the equality 
    $
    \gamma^2 = \lambda_1^2/\lambda_2^2 =  -\lambda_1/\lambda_2=1,
    $  
    and Lemma \ref{Lem:BC-Sub-critical flow} is satisfied. However, when $U_n > 0$, we have  
    $
    \gamma^2 = \lambda_1^2/\lambda_2^2 >  -\lambda_1/\lambda_2>1,
    $ 
    which violates Lemma \ref{Lem:BC-Sub-critical flow}. Thus when $U_n > 0$, the strictly outflow mass flux BC  \eqref{eq:BC-sub-critical_massflux_outflow} is not well-posed in the sense of Theorem \ref{theo:well-posdeness-ibvp}.
}}
{\example[Linear Bernoulli's principle]{
The Bernoulli's principle can be used to enforce how pressure or energy is  transported across the boundaries of the computational domain.  For the inflow boundary with $U_n<0$ we formulate the inflow  Bernoulli's BC
\begin{eqnarray}\label{eq:BC-sub-critical_velocity-flux_inflow}
\mathcal{B}\textbf{q}=
 \begin{cases}
    G_n := U_nu_n + gh= d_1 \equiv w_2 - \gamma w_1 = d_1/c, \quad \gamma = \lambda_1/\lambda_2 = \frac{U_n + c}{U_n-c},\\
   G_s := u_s = d_2 \equiv  \ {w_3 = {\frac{d_2}{c}}},
    \end{cases}
    \quad \text{if} \ U_n < 0.
    \end{eqnarray}
    Note that if $U_n < 0$, then $\gamma^2 =\lambda_1^2/\lambda_2^2 \le -\lambda_1/\lambda_2 < 1$. 
    Therefore the inflow BCs \eqref{eq:BC-sub-critical_velocity-flux_inflow} satisfy Lemma \ref{Lem:BC-Sub-critical flow},  and will result to a well-posed IBVP in the sense of Theorem \ref{theo:well-posdeness-ibvp}, when coupled to linear RSWE in vector in variant form \eqref{eqn:LinSWE2D}.
    
    When $U_n\ge 0$ the outflow Bernoulli's BC is
    \begin{eqnarray}\label{eq:BC-sub-critical_velocity-flux_outflow}
\mathcal{B}\textbf{q}=
    G_n := U_nu_n + gh= d_1 \equiv w_2 - \gamma w_1 = d_1/c, \quad \gamma = \lambda_1/\lambda_2 = \frac{U_n + c}{U_n-c},
    \quad \text{if} \ U_n \ge  0.
    \end{eqnarray}
    Thus, if $U_n \ge 0$, then $\gamma^2 = \lambda_1^2/\lambda_2^2 \ge   -\lambda_1/\lambda_2 \ge 1.$
   Again when $U_n = 0$, we have the equality 
    $
    \gamma^2 = \lambda_1^2/\lambda_2^2 =  -\lambda_1/\lambda_2=1,
    $  
    and Lemma \ref{Lem:BC-Sub-critical flow} is satisfied. However,
    When $U_n > 0$, we have  
    $
    \gamma^2 = \lambda_1^2/\lambda_2^2 >  -\lambda_1/\lambda_2>1,
    $ 
    which violates Lemma \ref{Lem:BC-Sub-critical flow}. Thus when $U_n > 0$, the strictly outflow Bernoulli's potential BC  \eqref{eq:BC-sub-critical_velocity-flux_outflow} is not well-posed in the sense of Theorem \ref{theo:well-posdeness-ibvp}.
}}

As shown in the examples above, note that at the inflow boundary with $U_n < 0$ or when $U_n=0$ we have much more flexibility, where the three physical BCs \eqref{eq:BC-sub-critical_Riemann_inflow}, \eqref{eq:BC-sub-critical_massflux_inflow} and \eqref{eq:BC-sub-critical_velocity-flux_outflow} yield stable and well-posed BCs. There is, however, less flexibility at the strictly outflow boundary with $U_n > 0$, since only the linear Riemann BC \eqref{eq:BC-sub-critical_Riemann_inflow} yield a stable and well-posed IBVP there.
\section{Nonlinear theory of IBVP for the RSWE}\label{sec:nonlinear-IBVP-RSWE}
In this section, we extend the linear analysis from the previous section to the nonlinear vector invariant RSWE \eqref{eqn:nlSWE2D}. It is particularly important to emphasize that the nonlinear analysis must be consistent with the conclusions derived from the linear analysis. Specifically, the number of boundary conditions at the inflow and outflow boundaries must align with the linear theory. Any discrepancy would imply that a valid linearization contradicts the linear analysis, thereby undermining the effectiveness of the nonlinear theory. 
A key and overarching requirement for the nonlinear IBVP is that the nonlinear boundary conditions be formulated such that their linearization results in a well-posed linear IBVP, in accordance with Theorem \ref{theo:well-posdeness-ibvp}. This consistency is essential for establishing the existence and stability of a unique, smooth solution.
\subsection{Nonlinear stability and linear consistency}
For the nonlinear vector invariant RSWE \eqref{eqn:nlSWE2D} our main objective is to design the nonlinear BCs such that we can prove total energy/entropy stability.
To begin, we consider the outward normal unit vector $\mathbf{n} = [n_x, n_y]^T$ and  introduce the variables
\begin{align}\label{eq:boundary-fields}
    u_n = n_xu + n_y v, \quad \quad u_s = n_yu - n_x v, \quad F_n = hu_n, \quad G_n = \frac12 u_n^2 + gh^2,
\end{align}
where $u_n$ is the normal velocity on the boundary, $u_s$ is the tangential velocity on the boundary.
We can rewrite the right hand side of the evolution equation of the total energy/entropy $E(t)$ given by \eqref{eq:energy_conservation} as
\begin{align}\label{eq:energy_conservation_2}
 \frac{d}{dt}E(t) =\mathrm{BT}:= -\oint_{\partial \Omega} G F_n dS, \quad G F_n = F_nG_n + \frac{1}{2} hu_n u_s^2,
\end{align}
where $\mathrm{BT}$ is the boundary term.
The following definition is crucial for the upcoming nonlinear analysis
\begin{definition}
    A nonlinear BC $\mathcal{B}\mathbf{q}=\mathbf{d}$ for the nonlinear vector invariant RSWE \eqref{eqn:nlSWE2D} for subcritical flows is energy/entropy stable if for homogeneous boundary data $\mathbf{d} =0$ we have $dE(t)/dt = \mathrm{BT} \le 0$, where $E(t)$ is the total energy/entropy and $\mathrm{BT}$ is the boundary term.
\end{definition}
\noindent
As before, the stability and well-posedness of the nonlinear IBVP are closely linked to the boundary term $\mathrm{BT}$. A necessary requirement is that the nonlinear BCs must ensure that the boundary term is never positive, that is $\mathrm{BT}\le 0$. However, unlike the linear case, for the nonlinear problem, the minimal number of BCs required to ensure $\mathrm{BT} \le 0$ may not suffice to guarantee well-posedness. Such an approach could also contradict the linear theory. For instance, natural boundary conditions such as no mass flux ($F_n=0$) or no-slip ($u_n=0$) lead to $\mathrm{BT} = 0$ regardless of the flow conditions. As shown in our linear analysis (see Table \ref{tab:num_bound}), at an inflow boundary, a single BC is insufficient to produce a well-posed IBVP for the linear RSWE. Therefore, the following two definitions are crucial for the current study.
%
\begin{definition}\label{def:linear-consistency}
    A nonlinear BC $\mathcal{B}\mathbf{q}=0$ for the nonlinear vector invariant RSWE \eqref{eqn:nlSWE2D} for subcritical flows is linearly consistent if the number of BCs specified by the nonlinear boundary  operator $\mathcal{B}\mathbf{q}$, at an inflow boundary and an outflow boundary, is consistent with the number of BCs prescribed by the linear theory, as  summarized in Table \ref{tab:num_bound}.
\end{definition}
\begin{definition}\label{def:linear-stability}
    A nonlinear BC $\mathcal{B}\mathbf{q}=0$ for the nonlinear vector invariant RSWE \eqref{eqn:nlSWE2D} for subcritical flows is linearly stable if a linearization of the boundary  operator $\mathcal{B}\mathbf{q}$ satisfies  Lemma \ref{Lem:BC-Sub-critical flow}.
\end{definition}
\noindent
If the nonlinear boundary operator $\mathcal{B}\mathbf{q}$ is both linearly consistent and linearly stable then the following theorem ensures the existence and stability of a smooth unique solution, for a sufficiently smooth and compactly supported initial data, and finite time $t\in[0,T]$.
\begin{theorem}\label{theo:exitenceofuniquesmoothsolution}
    Consider the nonlinear vector invariant RSWE \eqref{eqn:nlSWE2D} at subcritical flows subject to the BCs $\mathcal{B}\mathbf{q}=\mathbf{d}$ and the initial condition \eqref{eq:initial-condition} with compatible initial data $\mathbf{q}_0=[h_0(x,y), u_0(x,y), v_0(x,y)]^T\in \mathbb{R}^3$. Let $[H, U, V]^T \in\mathbb{R}^3$ be an arbitrary constant state with $H>0$ and $\mathbf{q} = [H + \widetilde{h}, U + \widetilde{u}, V + \widetilde{v}]^T$ linearizes the IBVP, that is \eqref{eqn:nlSWE2D}--\eqref{eq:initial-condition}.  If the boundary operator $\mathcal{B}\mathbf{q}=0$  is linearly consistent and linearly stable, then for every compactly supported and smooth initial data $\widetilde{\mathbf{q}}_0=[\widetilde{h}_0,\widetilde{u}_0,\widetilde{v}_0]^T$ there exists a unique solution $\widetilde{\mathbf{q}} = [\widetilde{h}, \widetilde{u}, \widetilde{v}]^T$
    that satisfies the estimate,
$$
\|\widetilde{\mathbf{q}}\|_W \leq \|\widetilde{\mathbf{q}}_0\|_W, \quad \forall \, t \in [0, T].
$$
\end{theorem}
\begin{proof}
    The proof can be adapted from the linear analysis  perform in the last section, in particular, from the proof of Theorem \ref{theo:well-posdeness-ibvp}.
\end{proof}
Theorem \ref{theo:exitenceofuniquesmoothsolution} can be used to establish the well-posedness of a nonlinear IBVP for sufficiently regular initial and boundary data and finite time $t \in [0, T]$.
\subsection{Stable nonlinear BCs.}
We will now formulate nonlinearly stable and linearly consistent BCs for the nonlinear vector invariant RSWE \eqref{eqn:nlSWE2D}. The discussion here is inspired by the 1D result and analysis performed in \cite{HEW2025113624}.
Let $F_n$  and  $G_n$ denote the nonlinear normal mass flux and Bernoulli's potential defined in \eqref{eq:boundary-fields}.
For $u_n< 0$ we formulate the inflow BC
\begin{eqnarray}\label{eq:BC-sub-critical_nonlinear_inflow}
\mathcal{B}\textbf{q}=
 \begin{cases}
     \alpha G_n - \beta F_n= d_1,   \\
    u_s = d_2,
   \\
    \end{cases}
    \ \text{if} \ u_n < 0.
    \end{eqnarray}
    When $u_n \ge 0$ we have the outflow BC
    \begin{eqnarray}\label{eq:BC-sub-critical_nonlinear_outflow}
\mathcal{B}\textbf{q}=
    \alpha G_n - \beta F_n= d_1,   \quad  \text{if} \ u_n \ge  0.
    \end{eqnarray}
Note that by construction the nonlinear BCs \eqref{eq:BC-sub-critical_nonlinear_inflow}--\eqref{eq:BC-sub-critical_nonlinear_outflow} are linearly consistent, that is there are two boundary conditions at the inflow boundary, with $u_n<0$, and one boundary condition at the outflow boundary, with $u_n\ge 0$.
Here the parameters $\alpha$ and $\beta$ are real and nonlinear weights which will be determined by ensuring entropy stability. The following Lemma constraints the boundary parameters $\alpha$ and $\beta$.
\begin{lemma}\label{Lem:BC-Sub-critical-nonlinear-flow}
    Consider the nonlinear boundary term $\operatorname{BT}$ defined in \eqref{eq:energy_conservation_2} and the BCs \eqref{eq:BC-sub-critical_nonlinear_inflow}--\eqref{eq:BC-sub-critical_nonlinear_outflow} with homogeneous boundary data $\mathbf{d} =0$. For sub-critical flows $ 0\le \operatorname{Fr} = |\mathbf{u}|/\sqrt{gh} < 1$, if  $\alpha\ge 0$, $\beta\ge0$, and $|\alpha| + |\beta| >0$  then the boundary term is never positive, that is $\mathrm{BT} \le 0$.
\end{lemma}
\begin{proof}
Note that if  $\alpha G_n - \beta F_n=0$ with $\beta=0$, $\alpha>0$ or $\alpha =0$, $\beta>0$ then $F_nG_n=0$. 
We consider first an inflow boundary with $u_n<0$ and the BC \eqref{eq:BC-sub-critical_nonlinear_inflow}. So if $u_s =0$ and $\beta=0$, $\alpha>0$ or $\alpha =0$, $\beta>0$, then we have $\mathrm{BT}=0$.
Now assume that $\alpha>0$ and $\beta \ge 0$, then we have $G_n = (\beta/\alpha) F_n \implies F_nG_n = (\beta/\alpha) F_n^2 \ge 0$. Thus for the inflow BC \eqref{eq:BC-sub-critical_nonlinear_inflow}, we have 
$$
\mathrm{BT}:= -\oint_{\partial \Omega} \left(F_nG_n + \frac{1}{2} hu_n u_s^2\right) dS=  -\oint_{\partial \Omega} \frac{\beta}{\alpha} |F_n|^2 dS\le 0.
$$
We now turn our attention to the outflow BC \eqref{eq:BC-sub-critical_nonlinear_outflow} with $u_n\ge 0$.
Again if $\beta=0$ or $\alpha =0$ then $F_nG_n=0$, and we have 
$$
\mathrm{BT}:= -\oint_{\partial \Omega} \left(F_nG_n + \frac{1}{2} hu_n u_s^2\right) dS=  -\oint_{\partial \Omega}\left(\frac{1}{2} hu_n u_s^2\right)dS\le 0.
$$
Finally, if $\alpha>0$ and $\beta \ge 0$ then  the outflow BC \eqref{eq:BC-sub-critical_nonlinear_outflow} with $u_n \ge 0$ gives
$$
\mathrm{BT}:= -\oint_{\partial \Omega} \left(F_nG_n + \frac{1}{2} hu_n u_s^2\right) dS=  -\oint_{\partial \Omega}\left( \frac{\beta}{\alpha} |F_n|^2 + \frac{1}{2} hu_n u_s^2\right)dS\le 0.
$$
The proof is complete.
\end{proof}
By using Lemma \ref{Lem:BC-Sub-critical-nonlinear-flow} we can proof the following theorem which ensures the energy/entropy stability of the nonlinear IBVP for the nonlinear vector invariant RSWE \eqref{eqn:nlSWE2D} at subcritical flows.
\begin{theorem}\label{theo:energy-stability-nonlinear-ibvp}
Consider the nonlinear IBVP defined by the nonlinear vector invariant RSWE \eqref{eqn:nlSWE2D} at subcritical flows subject to the BCs $\mathcal{B}\mathbf{q}=\mathbf{d}$,  and the initial condition \eqref{eq:initial-condition}. Let the nonlinear boundary operator  $\mathcal{B}\mathbf{q} $ be given by \eqref{eq:BC-sub-critical_nonlinear_inflow}--\eqref{eq:BC-sub-critical_nonlinear_outflow}, with homogeneous boundary data $\mathbf{d} =0$. If  $\alpha\ge 0$, $\beta\ge0$, and $|\alpha|^2 + |\beta|^2 >0$  then 
$$
\frac{d}{dt}E(t) =\mathrm{BT} \le 0 \iff E(t) \le E(0), \quad \forall\, t\ge0.
$$
\end{theorem}
\begin{proof}
    The proof follows from the evolution equation of the total energy/entropy $E(t)$ given by \eqref{eq:energy_conservation}  and \eqref{eq:energy_conservation_2}, that is $dE(t)/dt = \mathrm{BT}$. Subsequently, Lemma \ref{Lem:BC-Sub-critical-nonlinear-flow} ensures $\mathrm{BT} \le 0$ and finally time integration completes the proof.   
\end{proof}
Theorem \ref{theo:energy-stability-nonlinear-ibvp} establishes the stability of the solutions of the nonlinear RSWE IBVP  at subcritical flows, and seems to be analogous to Theorem \ref{theo:well-posdeness-ibvp} which proves the well-posedness of the linear IBVP. However, unlike the nonlinear analogue Theorem  \ref{theo:energy-stability-nonlinear-ibvp}, the linear result Theorem  \ref{theo:well-posdeness-ibvp} is comprehensive and establishes the existence and stability of a unique solution for the linear IBVP.

For the nonlinear IBVP, Theorem \ref{theo:exitenceofuniquesmoothsolution}, can be used to establish the existence and stability of a unique smooth solution for a reasonably regular initial and boundary data. However, Theorem \ref{theo:exitenceofuniquesmoothsolution} relies on the linear theory, Theorem  \ref{theo:well-posdeness-ibvp}, and will require the nonlinear boundary operator $\mathcal{B}\mathbf{q}$ to be  linearly consistent and linearly stable, see Definitions \ref{def:linear-consistency}--\ref{def:linear-stability}. As above, we will give some physically relevant examples of nonlinear BCs and use Theorem  \ref{theo:energy-stability-nonlinear-ibvp} to establish nonlinear stability.
{\example[Nonlinear Riemann invariants]{
We consider the nonlinear Riemann invariants, which are natural carriers of information in the system, and use them to formulate well-posed BCs. Here the nonlinear Riemann invariants on the boundaries are
\begin{align}
  r_1:= 2\sqrt{gh} + u_n, \quad r_2 := 2\sqrt{gh} - u_n,  \quad r_3 := u_s.  
\end{align}
Using the nonlinear Riemann invariants we will now formulate the nonlinearly stable and linearly consistent BCs for the nonlinear vector invariant RSWE \eqref{eqn:nlSWE2D}.
For $u_n< 0$ the inflow BC is given by
\begin{eqnarray}\label{eq:BC-Riemann_nonlinear_inflow}
\mathcal{B}\textbf{q}=
 \begin{cases}
   r_2 := 2\sqrt{gh} - u_n =d_1 \iff   \alpha G_n - \beta F_n= d_1,   \\
   r_3 := u_s = d_2,
   \\
    \end{cases}
    \ \text{if} \ u_n < 0.
    \end{eqnarray}
    When $u_n \ge 0$ we have the outflow BC
    \begin{eqnarray}\label{eq:BC-Riemann_nonlinear_outflow}
\mathcal{B}\textbf{q}=
    r_2 := 2\sqrt{gh} - u_n =d_1 \iff   \alpha G_n - \beta F_n= d_1,   \quad  \text{if} \ u_n \ge  0.
    \end{eqnarray}
Here the nonlinear coefficients $\alpha$, $\beta$ are
\begin{align}
    \alpha = \frac{2}{c}>0, \quad \beta = \frac{c+u_n}{hc} =\frac{1}{h} \left(1 +\frac{u_n}{c} \right), \quad c =\sqrt{gh} >0.
\end{align}
For subcritical flows  with $|u_n|/c  < \operatorname{Fr}= |\vb{u}|/c < 1$ we must have $\alpha, \beta > 0$.  The nonlinear boundary operator $\mathcal{B}\textbf{q}$ defined by \eqref{eq:BC-Riemann_nonlinear_inflow}--\eqref{eq:BC-Riemann_nonlinear_outflow} satisfy both Lemma \ref{Lem:BC-Sub-critical-nonlinear-flow} and Theorem \ref{theo:energy-stability-nonlinear-ibvp}. Therefore the nonlinear Riemann invariant BCs \eqref{eq:BC-Riemann_nonlinear_inflow}--\eqref{eq:BC-Riemann_nonlinear_outflow} are entropy stable at both the inflow  and outflow  boundaries.
By construction the BCs \eqref{eq:BC-Riemann_nonlinear_inflow}--\eqref{eq:BC-Riemann_nonlinear_outflow} are linearly consistent . In Example 1, we have shown that  the linearized boundary operator, the linear Riemann invariant, is  stable at both the inflow ($u_n <0$) and outflow ($u_n \ge 0$) boundaries. Thus the nonlinear BCs \eqref{eq:BC-Riemann_nonlinear_inflow}--\eqref{eq:BC-Riemann_nonlinear_outflow} are also linearly stable. We can then invoke Theorem \ref{theo:exitenceofuniquesmoothsolution} to prove the existence and stability of a unique smooth solution for sufficiently regular initial and boundary data and  a sufficiently short time interval.
}}
{\example[Nonlinear mass flux]{
The mass flux can be used to control how mass and materials are transported across the boundaries of the computational domain. For example the no mass flux BC can used to ensure that no mass is transported across the boundary. The nonlinear mass flux corresponds to setting $\alpha =0$ and $\beta = 1$ in  \eqref{eq:BC-sub-critical_nonlinear_inflow}--\eqref{eq:BC-sub-critical_nonlinear_outflow}, which gives the mass flux equation $-u_nh = d_1$, where $d_1$ is the mass flux data for the boundary. For $u_n<0$ we formulate the inflow  mass flux BC
\begin{eqnarray}\label{eq:BC-sub-critical_massflux_nonlinear-inflow}
\mathcal{B}\textbf{q}:=
 \begin{cases}
     -u_nh = d_1, \\
    u_s = d_2,
    \end{cases}
    \ \text{if} \ u_n < 0.
    \end{eqnarray}
    When $u_n\ge 0$  the outflow mass flux BC is
    \begin{eqnarray}\label{eq:BC-sub-critical_massflux_nonlinear-outflow}
\mathcal{B}\textbf{q}:= -u_nh = d_1 ,
    \ \text{if} \ u_n \ge  0.
    \end{eqnarray}
     Thus the nonlinear boundary operator $\mathcal{B}\textbf{q}$ defined by the nonlinear mass flux BCs \eqref{eq:BC-sub-critical_massflux_nonlinear-inflow}--\eqref{eq:BC-sub-critical_massflux_nonlinear-outflow} satisfy Lemma \ref{Lem:BC-Sub-critical-nonlinear-flow} and Theorem \ref{theo:energy-stability-nonlinear-ibvp}. Therefore the nonlinear mass flux BCs \eqref{eq:BC-sub-critical_massflux_nonlinear-inflow}--\eqref{eq:BC-sub-critical_massflux_nonlinear-outflow} are entropy stable at both the inflow  and outflow  boundaries. As above, by construction the BCs \eqref{eq:BC-sub-critical_massflux_nonlinear-inflow}--\eqref{eq:BC-sub-critical_massflux_nonlinear-outflow} are linearly consistent.
    In Example 2, we have shown that  the linearized boundary operators, linear mass flux, are stable at the inflow boundary and unstable at the outflow boundary. 
    Therefore the nonlinear inflow BC \eqref{eq:BC-sub-critical_massflux_nonlinear-inflow} is linearly stable and the nonlinear outflow BC \eqref{eq:BC-sub-critical_massflux_nonlinear-outflow} is linearly unstable.
    It is significantly noteworthy that the nonlinear IBVP, \eqref{eqn:nlSWE2D} with \eqref{eq:BC-sub-critical_massflux_nonlinear-inflow}--\eqref{eq:BC-sub-critical_massflux_nonlinear-outflow}, is  nonlinearly  energy/entropy stable and supports more general solutions at the inflow boundary with $u_n<0$  and at the outflow boundary with $u_n \ge 0$. 
}}
{\example[Nonlinear Bernoulli's principle]{
Bernoulli's principle can be used to enforce how pressure or energy is transmitted across the boundaries of the computational domain. The nonlinear Bernoulli's boundary condition corresponds to setting $\alpha =1$ and $\beta = 0$ in  \eqref{eq:BC-sub-critical_nonlinear_inflow}--\eqref{eq:BC-sub-critical_nonlinear_outflow} which gives the Bernoulli's equation $\frac{1}{2}u_n^2 + gh = d_1 $ where $d_1$ is the pressure energy data for the boundary. 
For $u_n<0$ we formulate the inflow  Bernoulli's  BC
\begin{eqnarray}\label{eq:BC-sub-critical_velocity-flux_nonlinear-inflow}
\mathcal{B}\textbf{q}:=
 \begin{cases}
     \frac{1}{2}u_n^2 + gh = d_1 \\
    u_s = d_2,
    \end{cases}
    \ \text{if} \ u_n < 0.
    \end{eqnarray}
    When $u_n\ge 0$ the outflow velocity flux BC is
    \begin{eqnarray}\label{eq:BC-sub-critical_velocity-flux_nonlinear-outflow}
\mathcal{B}\textbf{q}:=
    \frac{1}{2}u_n^2 + gh = d_1 ,
    \ \text{if} \ u_n \ge  0.
    \end{eqnarray}
     The nonlinear boundary operator $\mathcal{B}\textbf{q}$ defined by the nonlinear Bernoulli's  BC \eqref{eq:BC-sub-critical_velocity-flux_nonlinear-inflow}--\eqref{eq:BC-sub-critical_velocity-flux_nonlinear-outflow} satisfies Lemma \ref{Lem:BC-Sub-critical-nonlinear-flow} and Theorem \ref{theo:energy-stability-nonlinear-ibvp}. Therefore the nonlinear Bernoulli's BCs \eqref{eq:BC-sub-critical_velocity-flux_nonlinear-inflow}--\eqref{eq:BC-sub-critical_velocity-flux_nonlinear-outflow} are entropy stable at both the inflow  and outflow  boundaries. As above, by construction the BCs \eqref{eq:BC-sub-critical_velocity-flux_nonlinear-inflow}--\eqref{eq:BC-sub-critical_velocity-flux_nonlinear-outflow} are linearly consistent.
    In Example 3, we have shown that  the linearized boundary operators, linear Bernoulli's principle, are stable at the inflow and unstable at the outflow.
    It is also noteworthy that the nonlinear IBVP, \eqref{eqn:nlSWE2D} with \eqref{eq:BC-sub-critical_velocity-flux_nonlinear-inflow}--\eqref{eq:BC-sub-critical_velocity-flux_nonlinear-outflow}, is  nonlinearly  energy/entropy stable and supports more general solutions at the inflow boundary with $u_n<0$  and at the outflow boundary with $u_n \ge 0$.
}}
\subsection{Effects of the  Coriolis force term on remote boundary data}\label{sec:EffectsofCoriolis}
The RSWE incorporate the effects of a rotating frame of reference, primarily by including the Coriolis force term $f_c\ne 0$ in the momentum equations, which is absent in the standard SWE. While the Coriolis force may not directly influence the well-posedness or stability of the nonlinear IBVP solutions, for nonzero background flow velocities, this term can significantly modify the remote boundary data. Failure to account for the Coriolis effect in such cases may result in substantial mismatches, leading to errors that can contaminate the entire solution. We will now analyze the impact of the Coriolis force on the boundary data.

%
To begin, at $t=0$, for the nonlinear RSWE IBVP we consider the initial data
$$
\vb{q}_0(x,y) = [H+\widetilde{h}_0(x,y), U+\widetilde{u}_0(x,y), V+\widetilde{v}_0(x,y)]^\top, \quad (x,y) \in \Omega,
$$
where $H, \mathbf{U} = [U, V]^T$ are constant background states and $[\widetilde{h}_0(x,y), \widetilde{u}_0(x,y), \widetilde{v}_0(x,y)]^\top$ are  local smooth perturbations of the constant  background states which are compactly supported in $(x,y) \in \Omega$. In particular, the perturbations and their derivatives vanish completely at the  boundaries $\partial \Omega$ and at the far-field. We introduce the remote data ${H}_{\infty}, \mathbf{U}_{\infty}$ which is spatially invariant in $\Omega$ and satisfies the initial condition ${H}_{\infty}(0)=H, \mathbf{U}_{\infty}(0) = \mathbf{U}$.
The remote data solves the ordinary differential equation (ODE)
\begin{equation}\label{eq:remote-ODE}
    \begin{cases}
        \frac{d {H}_{\infty}}{d t} =0,\quad {H}_{\infty}(0) ={H},
        \\
 \frac{d \mathbf{U}_{\infty}}{d t} + f_c \mathbf{U}_{\infty}^\perp = \mathbf{0}, \quad \mathbf{U}_{\infty}(0) = \mathbf{U}= [U, V]^\top.
    \end{cases}
\end{equation}
Note that ${H}_{\infty}(t) = {H}$ for all $t \ge 0$, and  when $f_c=0$, we also have $\mathbf{U}_{\infty}(t) = \mathbf{U}$ for all $t \ge 0$.
Therefore for the standard nonlinear SWE with $f_c=0$ the remote velocity data does not change with time.
However for the nonlinear RSWE, when $f_c \ne 0$, the remote velocity data is given by the solution of the ODE \eqref{eq:remote-ODE}, 
$$
{H}_{\infty}(t) = {H}, \quad U_{\infty}(t) = U\cos(f_ct) + V\sin(f_ct), 
\quad
V_{\infty}(t) = V\cos(f_ct) - U\sin(f_ct).
$$
Thus, if the background flow velocity is zero, $\vb{U} = 0$, we also have a zero remote flow velocity data, $\vb{U}_{\infty}= \vb{U}=0$, for all $t\ge 0$. Similarly, for the standard nonlinear SWE with $f_c=0$, we again have $\mathbf{U}_{\infty}(t) = \mathbf{U}$ for all $t \ge 0$.

When the background flow velocity is nonzero, $\vb{U} \ne 0$, we introduce the rotated variables
$$
\overline{U}_n(t) = n_xU_{\infty}(t) + n_yV_{\infty}(t), 
\quad 
\overline{U}_s(t) = n_yU_{\infty}(t) - n_xV_{\infty}(t),
$$
where $\vb{n} = [n_x, n_y]^\top$ is the unit normal vector on the boundary, and $\overline{U}_n(t)$ is the normal component and $\overline{U}_s(t)$ is the tangential component of the remote velocity data.

To ensure consistency and avoid mismatch of boundary data, the boundary operator must be  matched to the remote boundary data. For instance, the nonlinear boundary operator $\mathcal{B}\textbf{q} = \textbf{d}$ defined by nonlinear Riemann invariants \eqref{eq:BC-Riemann_nonlinear_inflow}--\eqref{eq:BC-Riemann_nonlinear_outflow} gives
\begin{equation}\label{eq:BC-Riemann_nonlinear_data}
\begin{cases}
    2\sqrt{gh} - u_n =d_1 = 2\sqrt{gH} - \overline{U}_n, \\
    u_s = d_2  = \overline{U}_s.
\end{cases}
\end{equation}
For  the nonlinear mass flux BC \eqref{eq:BC-sub-critical_massflux_nonlinear-inflow}--\eqref{eq:BC-sub-critical_massflux_nonlinear-outflow} we have
\begin{equation}\label{eq:BC-sub-critical_mass-flux_nonlinear-data}
\begin{cases}
-u_nh = d_1 = - \overline{U}_nH,\\
 u_s = d_2  = \overline{U}_s.
 \end{cases}
\end{equation}
The nonlinear Bernoulli's  BC \eqref{eq:BC-sub-critical_velocity-flux_nonlinear-inflow}--\eqref{eq:BC-sub-critical_velocity-flux_nonlinear-outflow} gives
\begin{equation}\label{eq:BC-sub-critical_velocity-flux_nonlinear-data}
\begin{cases}
\frac{1}{2}u_n^2 + gh =   d_1 = \frac{1}{2}\overline{U}_n^2 + gH,\\
 u_s = d_2  = \overline{U}_s.
 \end{cases}
\end{equation}
In the next section we will derive numerical approximations of the linear and nonlinear IBVPs and prove numerical stability.

\section{Numerical method and analysis}\label{sec:numerical-analysis}
We will now present the numerical methods to solve the linear and nonlinear IBVPs on geometrically complex spatial domains, $(x,y)\in \Omega$. We will introduce a structure preserving curvilinear grid transformation that maps the PDE system from the physical domain $(x,y)\in \Omega$ to a reference unit square $(q,r)\in [0, 1]^2$. We will use SBP finite difference operators \cite{BStrand1994,kreiss1974finite,svard2014review,fernandez2014review} to approximate the spatial derivatives in the reference domain. The BCs will be implemented weakly using SAT \cite{carpenter1994time,Mattsson2003}, and we choose penalty parameters such that the semi-discrete approximations satisfy energy estimates analogous to the continuous energy estimates.

\subsection{Structure preserving curvilinear transformation}
We assume that $\Omega$ is geometrically complex but sufficiently smooth such that there is an invertible map between $(x,y)\in \Omega$ and the unit square $(q,r)\in [0, 1]^2$, that is $(x,y) \leftrightarrow (q, r)$. Given a 2D scalar field $u(x,y)$, the  following curvilinear transformation identities hold for the partial derivatives
\begin{align}\label{eq:version1}
 u_x = q_x u_q + r_xu_r, \quad  u_y = q_y u_q + r_yu_r,
\end{align}
and
\begin{align}\label{eq:version2}
 u_x = \frac{1}{J} \left((Jq_x u)_q + (Jr_xu)_r\right), \quad  u_y = \frac{1}{J} \left((Jq_y u)_q + (Jr_yu)_r\right).
\end{align}
The subscripts denote partial derivatives. Here $q_x, r_x, q_y, r_y$ are metric derivatives and $J>0$ is the Jacobian of the curvilinear transformation given by the metric relations
\begin{align}
Jq_x = y_r, \quad Jr_x = -y_q, \quad Jq_y = -x_r, \quad Jr_y = x_q, \quad J = x_qy_r-x_ry_q >0.
\end{align}
The two identities, the non-conservative transformation \eqref{eq:version1} and the conservative transformation \eqref{eq:version2}, are equivalent at the continuous level.
However, they will give different discrete operators when discretized and discrete approximations of spatial derivatives are introduced.
Indeed, for  stable and consistent discrete approximations the  discrete derivative operators will converge to the continuous derivative operators.
\begin{definition}\label{def:structure-preserving}
    A curvilinear transformation of the nonlinear  RSWE \eqref{eqn:nlSWE2D} and the linear  RSWE \eqref{eqn:LinSWE2D}, using the non-conservative transformation \eqref{eq:version1} and the conservative transformation \eqref{eq:version2}, is called structure preserving if the evolution equations \eqref{eq:energy_conservation} and \eqref{eq:energy_conservation_linear} of the total energy/entropy can be derived using only integration by parts, without the chain rule or product rule.
\end{definition}
To ensure structure preservation we will transform the gradient operator $\grad$ using the non-conservative transformation \eqref{eq:version1}, and the divergence $\div{}$ and curl $\curl{}$  operators using the conservative transformation \eqref{eq:version2}. That is 
\begin{equation}\label{eq:curvilinear-transformation}
\begin{split}
    \grad{G} &= 
\begin{bmatrix}
    q_x G_q + r_xG_r\\
    q_y G_q + r_yG_r
\end{bmatrix},
\quad 
\div{\mathbf{F}}  = \frac{1}{J} \left((Jq_x F_u)_q + (Jr_xF_u)_r\right) +  \frac{1}{J} \left((Jq_y F_v)_q + (Jr_yF_v)_r\right),\\
\curl{\mathbf{F}}  &= \frac{1}{J} \left((Jq_x F_v)_q + (Jr_xF_v)_r\right) - \frac{1}{J} \left((Jq_y F_u)_q + (Jr_yF_u)_r\right).
\end{split}
\end{equation}
Here $G: \Omega \to \mathbb{R}$ is a scalar field and $\mathbf{F}: \Omega \to \mathbb{R}^2$ is a vector field.
For the nonlinear  RSWE \eqref{eqn:nlSWE2D} and the linear  RSWE \eqref{eqn:LinSWE2D} with the curvilinear transformations \eqref{eq:curvilinear-transformation}, we can show that the transformation is structure preserving. In particular, using only integration by parts without the chain/product rule, we can  show the  conservation properties of   total energy and total vorticity,  and steady linear geostrophic balance for the transformed continuous operator \cite{ricardo2023conservation}.

Our main focus here is to design energy/entropy stable numerical methods for the nonlinear and linear IBVPs for the RSWE.
As above,  the total energy/entropy $E(t)$ is the integral of the elemental energy/entropy over the spatial domain
\begin{equation}\label{eq:elemental_curv}
    E(t) = \int_{\Omega} e d\Omega = \int_{0}^1\int_{0}^1e Jdqdr,
\end{equation}
where the elemental energy $e$ is defined by \eqref{eq:elemental} for nonlinear the RSWE and by \eqref{eq:energy_linear} for the linear RSWE.
Define the boundary terms 
{\footnotesize
\begin{equation}\label{eq:BT-nonlinear}
\mathrm{BT}:= -\sum_{\xi = q, r}\int_{0}^1 \left.\left(\left(\sqrt{\xi_x^2 + \xi_y^2}J\left(\frac{1}{2}gh^2u_n + eu_n\right)\right)\right|_{\xi = 0} + \left.\left(J\sqrt{\xi_x^2 + \xi_y^2}\left(\frac{1}{2}gh^2u_n + eu_n\right)\right)\right|_{\xi = 1}\right) \frac{dqdr}{d\xi},
\end{equation}
}
for the transformed nonlinear RSWE \eqref{eqn:nlSWE2D} with the curvilinear transformations \eqref{eq:curvilinear-transformation}, and the boundary terms 
{\scriptsize
\begin{equation}\label{eq:BT-linear}
\begin{split}
&\mathrm{BT}:= -\sum_{\xi = q, r}\int_{0}^1 \left.\left(\left(\sqrt{\xi_x^2 + \xi_y^2}J\left(gHh{u}_n + e{U}_n\right)\right)\right|_{\xi = 0} + \left.\left(J\sqrt{\xi_x^2 + \xi_y^2}\left(gHh{u}_n + e{U}_n\right)\right)\right|_{\xi = 1}\right) \frac{dqdr}{d\xi}\\
&= -\frac{c^2 H}{2}\sum_{\xi = q, r}\int_{0}^1 \left.\left(\left(\sqrt{\xi_x^2 + \xi_y^2}J\left(\lambda_1 w_1^2 + \lambda_2 w_2^2 + \lambda_3 w_3^2\right)\right)\right|_{\xi = 0} + \left.\left(J\sqrt{\xi_x^2 + \xi_y^2}\left(\lambda_1 w_1^2 + \lambda_2 w_2^2 + \lambda_3 w_3^2\right)\right)\right|_{\xi = 1}\right) \frac{dqdr}{d\xi},   
\end{split}
\end{equation}
}
for the transformed linear the RSWE \eqref{eqn:LinSWE2D} with the curvilinear transformations \eqref{eq:curvilinear-transformation}.
    Using only integration by parts without the chain/product rule, we can so show that the evolution equation for the total  energy $E(t)$ is given by 
    \begin{equation}\label{eq:energy-rate}
      \frac{dE}{dt} = \mathrm{BT},
    \end{equation}
    where the boundary terms $\mathrm{BT}$ are defined by  \eqref{eq:BT-nonlinear} and  \eqref{eq:BT-linear}.  Next we will introduce discrete approximations of the transformed operators \eqref{eq:curvilinear-transformation} and try as much as possible to replicate the evolution equation \eqref{eq:energy-rate} of the total energy at the discrete level.
\subsection{Semi-discrete approximation}
We discretise the reference computational square $(q,r)\in   [0,1]^2$ with an evenly spaced mesh across each axis, $\xi \in \{ q, r \}$. 
    For each $\xi \in \{q,r\}$, consider the uniform discretisation of the unit interval $\xi \in [0, 1]$ 
    \begin{align}
        \xi_{i} = \frac{i-1}{n_{\xi}-1}, \quad i \in \{1, 2, \dots, n_{\xi} \}, 
    \end{align}
    where $n_{\xi}$ is the number of grid-points used on the $\xi$-axis. 
    
    We will use the traditional SBP operators \cite{BStrand1994,kreiss1974finite} to approximate the spatial derivatives, $\partial/{\partial\xi}$. 
	For each $\xi \in \{q,r\}$ define $H_{\xi} =\text{diag}\left(h_1^{(\xi)} , \dots, h_{n_{\xi}}^{(\xi)}\right)$, with $h_j^{(\xi)} >0$ for all $j \in \{1, \dots, n_\xi \}$.
	We consider the SBP derivative operators $D_{\xi} : \R^{n_{\xi}} \mapsto \R^{n_{\xi}}$ so that the SBP property holds
	\begin{align}\label{eq:upw_SBP}
        (D_{\xi} \bm{f} )^\top H_{\xi} \bm{g} +  \bm{f}^\top H_{\xi} (D_{\xi }\bm{g}) = f(\xi_{n_{\xi}})g(\xi_{n_{\xi}}) - f(\xi_1)g(\xi_1),
    \end{align}
    where  $\bm{f} = (f(\xi_1), \dots , f(\xi_{n_{\xi}}) )^\top$, $\bm{g} = (g(\xi_1), \dots , g(\xi_{n_{\xi}}) )^\top$ are vectors sampled from weakly differentiable functions of the $\xi$ variable. 
 %
%

	The 1D SBP operators can be extended to higher space dimensions using tensor products $\otimes$. Let $f:(q,r)\to \mathbb{R}$ denote a 2D scalar funtion, and $f_{ij} := {f}(q_i,r_j)$ denote the corresponding 2D grid function.
	The  2D scalar grid function $f_{ij}$ is rearranged row-wise as a vector $\bm{f}$ of length $n_qn_r$. For $\xi \in \{q,r\}$  define the 2D spatial discrete operators
     {
    \begin{align*}
\centering
&\bm{D}_{q} = \left( {D}_{ q} \otimes I_{n_r}\right), \quad \bm{D}_{r} = \left(I_{n_q} \otimes  {D}_{ r}\right), \quad
\bm{H} = \left( H_{q} \otimes H_r\right), 
\end{align*}
    where $I_{n_{\xi}}$ are the identity matrices of size $n_{\xi} \times n_{\xi}$. 
    }
    The matrix operator $\bm{D}_{\xi}$  will approximate the partial derivative operator in the $\xi$-direction. A discrete inner product on $\R^{n_{q} \times n_{r}}$ is induced by $\bm{H}$ through 
     \begin{align}
        \l \bm{g} , \bm{f} \r_{\bm{H}} := \bm{g}^\top \bm{H} \bm{f} = \sum_{i=1}^{n_q} \sum_{j=1}^{n_r} f_{ij}g_{ij} h_i^{(q)} h_j^{(r)}, 
    \end{align}
        and  the  discrete total energy   is given by
     \begin{equation}\label{eq:elemental_curv-discrete}
     E_d(t) = \sum_{i=1}^{n_q}\sum_{j=1}^{n_r} e_{ij}J_{ij} h_i^{(q)}h_j^{(r)}, 
\end{equation}
where $e_{ij}$ are the elemental energies sampled on the grid.

Next, we use the SBP operators to approximate the the transformed differential operators \eqref{eq:curvilinear-transformation}, and we have  
{\footnotesize
\begin{equation}\label{eq:curvilinear-transformation-discrete}
\begin{split} 
&\grad_d\cdot{\mathbf{F}}  = {J}^{-1} \left(\left(\mathbf{D}_q(Jq_x F_u) + \mathbf{D}_r(Jr_x F_u)\right) +   \left(\mathbf{D}_q(Jq_y F_v) + \mathbf{D}_r(Jr_y F_v)\right)\right),\\
&\grad_d{G} = 
\begin{bmatrix}
    q_x \mathbf{D}_qG + r_x\mathbf{D}_rG\\
    q_y \mathbf{D}_qG + r_y\mathbf{D}_rG
\end{bmatrix},
\quad
\grad_d\times{\mathbf{F}}  = {J}^{-1}\left( \left(\mathbf{D}_q(Jq_x F_v) + \mathbf{D}_r(Jr_xF_v)\right) -  \left(\mathbf{D}_q(Jq_y F_u) + \mathbf{D}_r(Jr_yF_u)\right)\right).
\end{split}
\end{equation}
}
The semi-discrete approximations of the nonlinear  RSWE \eqref{eqn:nlSWE2D} on the curvilinear mesh is derived by replacing the continuous operators \eqref{eq:curvilinear-transformation} with the discrete operators \eqref{eq:curvilinear-transformation-discrete}. We have
\begin{equation}\label{eqn:nlSWE2D-semi-discrete}
    \begin{cases}
    \frac{d h}{d t } + \grad_d\cdot{\mathbf{F}} = 0, \\
    \frac{d \mathbf{u}}{d t} + \omega \mathbf{u}^\perp +  \grad_d{G} = \mathbf{0},\\
    \mathbf{F}=h\mathbf{u}, \quad G = \frac{{1 }}{{2}}|\mathbf{u}|^2 +gh, \quad \omega = \grad_d\times{\mathbf{u}} + f_c.\end{cases}
\end{equation}
Similarly, we also approximate the linear  RSWE \eqref{eqn:LinSWE2D} with the discrete differential operators   \eqref{eq:curvilinear-transformation-discrete} on the grid, yielding 
\begin{equation}\label{eqn:LinSWE2D-semi-discrete}
    \begin{cases}
    \frac{d h}{d t } + \grad_d\cdot{\mathbf{F}} = 0, \\
    \frac{d \mathbf{u}}{d t} + \omega \mathbf{U}^\perp + f_c \mathbf{u}^\perp + \grad_d{G} = \mathbf{0},\\
     \mathbf{F}= H\mathbf{u} + \mathbf{U}h, \quad G = \mathbf{U}\cdot\mathbf{u} +gh, \quad\omega =  \grad_d\times{\mathbf{u}}.
    \end{cases}
\end{equation}
We approximate the boundary integrals \eqref{eq:BT-nonlinear}--\eqref{eq:BT-linear} by the numerical quadrature rules induced by $\bm{H}_{\xi}$, which gives the discrete the boundary terms
{\small
\footnotesize
\scriptsize
\begin{equation}\label{eq:BT-nonlinear-discrete}
\begin{split}
\mathrm{BT}_d:= &-\sum_{\xi = q, r}\sum_{j=1}^{\frac{n_qn_r}{n_\xi}} \left(\left.\left(J_j\sqrt{\xi_{xj}^2 + \xi_{yj}^2}\left(\frac{1}{2}gh_j^2u_{nj} + eu_{nj}\right)\right)\right|_{\xi_{1} = 0} 
+ \left.\left(J_j\sqrt{\xi_{xj}^2 + \xi_{yj}^2}\left(\frac{1}{2}gh_j^2u_{nj} + eu_{nj}\right)\right)\right|_{\xi_{n_\xi} = 1}\right) \frac{h_j^{(q)}h_j^{(r)}}{h_j^{(\xi)}},
\end{split}
\end{equation}
}
for the semi-discrete nonlinear RSWE \eqref{eqn:nlSWE2D-semi-discrete}. The semi-discrete linear RSWE \eqref{eqn:LinSWE2D-semi-discrete}  gives the discrete boundary terms  
{
\scriptsize
\begin{equation}\label{eq:BT-linear-discrete}
\begin{split}
&\mathrm{BT}_d:= -\sum_{\xi = q, r}\sum_{j=1}^{\frac{n_qn_r}{n_\xi}} \left.\left(\left(\sqrt{\xi_{xj}^2 + \xi_{yj}^2}J_j\left(gHh_j{u}_{nj} + e_j{U}_n\right)\right)\right|_{\xi_{1} = 0} + \left.\left(J_j\sqrt{\xi_{xj}^2 + \xi_{yj}^2}\left(gHh_j{u}_{nj} + e_j{U}_n\right)\right)\right|_{\xi_{n_\xi} = 1}\right) \frac{h_j^{(q)}h_j^{(r)}}{h_j^{(\xi)}},
\end{split}
\end{equation}
}
By using the eigen decomposition \eqref{eq:linear_transform} the boundary term defined in \eqref{eq:BT-linear-discrete}, for the semi-discrete linear RSWE, can be rewritten as
{
\begin{equation}\label{eq:BT-linear-discrete-eigen-decompose}
\begin{split}
    %
\mathrm{BT}_d= &-\frac{c^2H}{2}\sum_{\xi = q, r}\sum_{j=1}^{\frac{n_qn_r}{n_\xi}} \left(\left.\left(\sqrt{\xi_{xj}^2 + \xi_{yj}^2}J_j\left(\lambda_1w_{1j}^2 + \lambda_2w_{2j}^2 + \lambda_3w_{3j}^2\right)\right)\right|_{\xi_{1} = 0} \right)
\frac{h_j^{(q)}h_j^{(r)}}{h_j^{(\xi)}}
\\
&
-
\frac{c^2H}{2}\sum_{\xi = q, r}\sum_{j=1}^{\frac{n_qn_r}{n_\xi}} \left( \left.\left(J_j\sqrt{\xi_{xj}^2 + \xi_{yj}^2}\left(\lambda_1w_{1j}^2 + \lambda_2w_{2j}^2 + \lambda_3w_{3j}^2\right)\right)\right|_{\xi_{n_\xi} = 1}\right) \frac{h_j^{(q)}h_j^{(r)}}{h_j^{(\xi)}}.
\end{split}
\end{equation}
}
\begin{theorem}
    Consider the semi-discrete nonlinear  RSWE \eqref{eqn:nlSWE2D-semi-discrete} and the semi-discrete linear  RSWE \eqref{eqn:LinSWE2D-semi-discrete}, where the total semi-discrete energy $E_d(t)$ is given by \eqref{eq:elemental_curv-discrete}. If the discrete derivative operators satisfy the SBP property \eqref{eq:upw_SBP}, then   the evolution equation for the total semi-discrete  energy $E_d(t)$ is given by 
    $
    {dE_d}/{dt} = \mathrm{BT}_d,
    $
    where the boundary terms $\mathrm{BT}_d$ are given by  \eqref{eq:BT-nonlinear-discrete} and  \eqref{eq:BT-linear-discrete}.
\end{theorem}
\begin{proof}
    Since the continuous energy rate \eqref{eq:energy-rate} is derived by using integration by parts only it follows that SBP will yield similar evolution equation for the total discrete energy given by \eqref{eq:elemental_curv-discrete}.
\end{proof}
Note that we are yet to enforce the BCs \eqref{eq:BC-sub-critical_inflow}--\eqref{eq:BC-sub-critical_outflow} or \eqref{eq:BC-sub-critical_nonlinear_inflow}--\eqref{eq:BC-sub-critical_nonlinear_outflow} at the boundaries.

\subsection{Stable numerical boundary procedures}
We will now implement the BCs and prove numerical stability. The BCs will be imposed by adding SATs, source terms, which encode the boundary operators, to the semi-discrete approximations \eqref{eqn:nlSWE2D-semi-discrete} and  \eqref{eqn:LinSWE2D-semi-discrete} with penalty weights chosen to ensure numerical stability.

For boundaries in the $\xi$-axis, we make the ansatz 
$$
\operatorname{SAT}_\xi = 
\begin{bmatrix}
    \operatorname{SAT}^{h}_\xi\\
    \operatorname{SAT}^u_\xi
\end{bmatrix},
$$
where $ \left.\operatorname{SAT}^{h}_\xi\right|_{ij} \in \mathbb{R}$ are the $\operatorname{SAT}$s for the mass equation and  $ \left.\operatorname{SAT}^{u}_\xi\right|_{ij} \in \mathbb{R}^2$ are the $\operatorname{SAT}$s for the velocity equations. A semi-discrete approximation of the IBVP for the nonlinear RSWE is obtained by appending the $\operatorname{SAT}$s to the right hand sides of \eqref{eqn:nlSWE2D-semi-discrete}, we have 
\begin{equation}\label{eqn:nlSWE2D-semi-discrete-SAT}
    \begin{cases}
    \frac{d h}{d t } + \grad_d\cdot{\mathbf{F}} = \sum_{\xi = q, r}\operatorname{SAT}^{h}_\xi, \\
    \frac{d \mathbf{u}}{d t} + \omega \mathbf{u}^\perp +  \grad_d{G} = \sum_{\xi = q, r}\operatorname{SAT}^u_\xi,\\
    \mathbf{F}=h\mathbf{u}, \quad G = \frac{{1 }}{{2}}|\mathbf{u}|^2 +gh, \quad \omega = \grad_d\times{\mathbf{u}} + f_c.\end{cases}
\end{equation}
Similarly, for the linear RSWE we append the $\operatorname{SAT}$s to the right hand sides of \eqref{eqn:LinSWE2D-semi-discrete}, giving
\begin{equation}\label{eqn:LinSWE2D-semi-discrete-SAT}
    \begin{cases}
    \frac{d h}{d t } + \grad_d\cdot{\mathbf{F}} = \sum_{\xi = q, r}\operatorname{SAT}^{h}_\xi, \\
    \frac{d \mathbf{u}}{d t} + \omega \mathbf{U}^\perp + f_c \mathbf{u}^\perp + \grad_d{G} = \sum_{\xi = q, r}\operatorname{SAT}^{u}_\xi,\\
     \mathbf{F}= H\mathbf{u} + \mathbf{U}h, \quad G = \mathbf{U}\cdot\mathbf{u} +gh, \quad\omega =  \grad_d\times{\mathbf{u}}.
    \end{cases}
\end{equation}
The exact forms of the $\operatorname{SAT}$s and the penalties will be derived and analyzed below. To keep in mind, it is significantly important that the $\operatorname{SAT}$s ensure both consistency and numerical stability.
The definition of numerical stability will be important for the following analysis.
\begin{definition}
Consider the semi-discrete approximations \eqref{eqn:nlSWE2D-semi-discrete-SAT} and \eqref{eqn:LinSWE2D-semi-discrete-SAT}
    for subcritical flows, where the BCs are enforced weakly using $\operatorname{SAT}$.  The semi-discrete approximation \eqref{eqn:nlSWE2D-semi-discrete-SAT} or \eqref{eqn:LinSWE2D-semi-discrete-SAT} is called energy/entropy stable if for homogeneous boundary data $\mathbf{d} =0$ we have $dE_d(t)/dt = \mathbb{BT}_n \le 0$, where $E_d(t)$ is the total semi-discrete energy/entropy defined by \eqref{eq:elemental_curv-discrete}.
\end{definition}
We will begin with the $\operatorname{SAT}$s for the linear RSWE and proceed later to  the nonlinear RSWE.
\subsubsection{SATs for the linear RSWE}
Here, we will derive the $\operatorname{SATs}$ for the linear RSWE. At each point on the boundary we set the $\operatorname{SAT}$
\begin{equation}\label{eq:sat-linear}
\mathrm{SAT}^{n}_{j} = 
\left({W}^{-1}R{P}{S}\right)
\begin{bmatrix}
    \tau_h\left(w_2 - \gamma w_1- d_1\right)\\
     \tau_n\left(w_2 - \gamma w_1- d_1\right)\\
     0
\end{bmatrix},
\quad
\mathrm{SAT}^{s}_{j} = 
\left({W}^{-1}R{P}{S}\right)
\begin{bmatrix}
    0\\
    0\\
     \tau_s U_n\left(w_3- d_2\right)
\end{bmatrix}
\end{equation}
with the $3$-by-$3$ matrices given by
$$
 W =  \frac{1}{2}\begin{bmatrix}
g & 0 & 0\\
0 & H & 0\\
0 & 0 & H
\end{bmatrix}
,\quad 
S = \frac{1}{\sqrt{2}}\begin{bmatrix}
        1&1 & 0\\
        1&-1&0\\
        0 & 0 & \sqrt{2}
    \end{bmatrix}, \quad 
    {P}
     = \begin{bmatrix}
        \frac{1}{H}&0 & 0\\
        0&\frac{1}{c}&0\\
        0 & 0 & \frac{1}{c}
    \end{bmatrix},
    \quad
{R}
     = \begin{bmatrix}
        1 &0 & 0\\
        0&n_x&n_y\\
        0 & n_y & -n_x
    \end{bmatrix},
$$
where $\tau_h$, $\tau_n$ and $\tau_s$ are penalty parameters that will be determined by requiring numerical stability and $\bm{m} = [m_x, m_y]^\top = [n_y, -n_x]^\top$ is the tangential unit vector on the boundary. Note that $u_s = \bm{m}\cdot \bm{u}$ is the tangential component of the velocity field on the boundary. It is noteworthy that the $\mathrm{SAT}^{n}$ and $\mathrm{SAT}^{s}$ encode the boundary operators defined by \eqref{eq:BC-sub-critical_inflow}--\eqref{eq:BC-sub-critical_outflow} and would vanish if the BCs are satisfied exactly by $w_1, w_2, w_3$.

The following identities 
$$
\begin{bmatrix}
    h\\
    u\\
    v
\end{bmatrix}^\top W\mathrm{SAT}^{n}_{j} 
= 
\tau_h w_1 \left(w_2 - \gamma w_1- d_1\right) +
     \tau_n w_2\left(w_2 - \gamma w_1- d_1\right),
     \quad
\begin{bmatrix}
    h\\
    u\\
    v
\end{bmatrix}^T W\mathrm{SAT}^{s}_{j} 
= 
\tau_s w_3 U_n \left(w_3- d_1\right) ,
$$
will be useful for the stability analysis below. 

We introduce the unit vectors
$e_1 = [1, 0, \cdots, 0]^T\in \mathbb{R}^{n_\xi}$, $e_{n} = [0,  \cdots, 0, 1]^T\in \mathbb{R}^{n_\xi}$
and the boundary projection matrices
\begin{align}\label{eq:boundary-projection}
\centering
&\mathbf{e}_{q, 1} = \left( e_1 e_1^T \otimes I_{n_r}\right), \quad \mathbf{e}_{q, n_q} = \left( e_n e_n^T \otimes I_{n_r}\right), \quad \mathbf{e}_{r, 1} = \left(I_{n_q} \otimes  e_1 e_1^T\right), \quad
\mathbf{e}_{r, n_r} = \left(I_{n_q} \otimes  e_n e_n^T\right).
\end{align}
The $\xi$-axis SAT for the linear RSWE is given by
{\footnotesize
\begin{eqnarray}\label{eq:sat-axis-linear}
\operatorname{SAT}_\xi=
 \begin{cases}
    \frac{c^2H}{2}\bm{H}_{\xi}^{-1}\left(\mathbf{e}_{\xi, 1}\sqrt{\bm{\xi}_x^2 + \bm{\xi}_y^2}\left(\mathrm{SAT}^{n} + \mathrm{SAT}^{s}\right) + \mathbf{e}_{\xi, n_\xi}\sqrt{\bm{\xi}_x^2 + \bm{\xi}_y^2}\left(\mathrm{SAT}^{n} + \mathrm{SAT}^{s}\right)\right), \quad \ \text{if} \ U_n < 0, \\
   \frac{c^2H}{2}\bm{H}_{\xi}^{-1}\left(\mathbf{e}_{\xi,1}\sqrt{\bm{\xi}_x^2 + \bm{\xi}_y^2}\mathrm{SAT}^{n} + \mathbf{e}_{\xi,n_\xi}\sqrt{\bm{\xi}_x^2 + \bm{\xi}_y^2}\mathrm{SAT}^{n}\right), \quad \ \text{if} \ U_n > 0,
    \end{cases}
\end{eqnarray}
}
where $\mathrm{SAT}^{n}$ and $\mathrm{SAT}^{s}$ are given by \eqref{eq:sat-linear}.
Here $\mathbf{e}_{\xi,1}$ and  $\mathbf{e}_{\xi,n_\xi}$ defined in \eqref{eq:boundary-projection} are projection matrices that project the SATs to the respective boundaries $\xi_1 =0$ and $\xi_{n_\xi} = 1$. We assume homogeneous boundary data $\bm{d}=0$ and introduce the numerical boundary term at every point on the boundary
{\small
\begin{equation}\label{eq:boundary-term-pointwise-linear}
    \operatorname{\mathbb{BT}}_j =
\begin{cases}
-\left(\lambda_1w_{1j}^2 + \lambda_2w_{2j}^2 + \lambda_3w_{3j}^2\right)+ \tau_h w_{1j} \left(w_{2j} - \gamma w_{1j}\right) +
     \tau_n w_{2j}\left(w_{2j} - \gamma w_{1j}\right) + \tau_s U_n w_{3j}^2, \quad \ \text{if} \ U_n < 0,\\
     
-\left(\lambda_1w_{1j}^2 + \lambda_2w_{2j}^2 + \lambda_3w_{3j}^2\right)+ \tau_h w_{1j} \left(w_{2j} - \gamma w_{1j}\right) +
     \tau_n w_{2j}\left(w_{2j} - \gamma w_{1j}\right) , \quad \ \text{if} \ U_n \ge  0.
     \end{cases}
\end{equation}
}
The following Lemma constrains the penalty parameters and ensures that the boundary term $\operatorname{\mathbb{BT}}_j$  is never positive.
\begin{lemma}\label{lem:boundary-term-pointwise-linear}
    Consider the boundary term  $\operatorname{\mathbb{BT}}_j$ defined by \eqref{eq:boundary-term-pointwise-linear}, where $\lambda_1 = U_n + c >0$, $\lambda_2 = U_n-c < 0$ and $\lambda_3 = U_n$. If $\gamma^2 \le -\lambda_1/\lambda_2$, and $\tau_n =\lambda_2$, $\tau_h = \gamma \lambda_2$ and $\tau_s \ge 1$,  then the boundary term is never positive, that is $\operatorname{\mathbb{BT}}_j \le 0$.
\end{lemma}
\begin{proof}
    We consider the boundary term  $\operatorname{\mathbb{BT}}_j$ defined by \eqref{eq:boundary-term-pointwise-linear} and set $\tau_n =\lambda_2$, $\tau_h = \gamma \lambda_2$, we have
  \begin{equation*}
    \operatorname{\mathbb{BT}}_j =
\begin{cases}
-\left(\lambda_1 + \lambda_2 \gamma^2\right)w_{1j}^2 +  U_n (\tau_s-1) w_{3j}^2, \quad \ \text{if} \ U_n < 0\\
     
-\left(\lambda_1 + \lambda_2 \gamma^2\right)w_{1j}^2  - U_n w_{3j}^2
      , \quad \ \text{if} \ U_n \ge  0.
     \end{cases}
\end{equation*}  
Note that with $\lambda_1 > 0$ and $\lambda_2 < 0$ we have  $\gamma^2 \le -\lambda_1/\lambda_2 \implies \left(\lambda_1 + \lambda_2 \gamma^2\right) \ge 0$. If $\tau_s \ge 1$  then the boundary is never positive, that is $\operatorname{\mathbb{BT}}_j \le 0$.
\end{proof}
We will now state the theorem which ensures the stability of the linear semi-discrete approximation \eqref{eqn:LinSWE2D-semi-discrete-SAT}.
\begin{theorem}
    Consider the semi-discrete approximation  \eqref{eqn:LinSWE2D-semi-discrete-SAT} with the SAT \eqref{eq:sat-axis-linear}, for the vector invariant  linear RSWE \eqref{eqn:LinSWE2D} 
    at sub-critical flows, with $\lambda_1 = U_n + c >0$, $\lambda_2 = U_n-c < 0$ and $\lambda_3 = U_n$. If $\gamma^2 \le -\lambda_1/\lambda_2$, and $\tau_n =\lambda_2 < 0$, $\tau_h = \gamma \lambda_2$ and $\tau_s \ge 1$,  then the semi-discrete approximation  \eqref{eqn:LinSWE2D-semi-discrete-SAT} is energy stable. That is, with homogeneous boundary data $\bm{d}=0$, we have
    $$
    \frac{dE_d}{dt} = \mathbb{BT}_n:= \sum_{\xi = q, r}\sum_{j=1}^{\frac{n_qn_r}{n_\xi}} \left.\left(J_j\left(\sqrt{\xi_{xj}^2 + \xi_{yj}^2}\operatorname{\mathbb{BT}}_j\right)\right|_{\xi_{1} = 0} + \left.\left(J_j\sqrt{\xi_{xj}^2 + \xi_{yj}^2}\operatorname{\mathbb{BT}}_j\right)\right|_{\xi_{n_\xi} = 1}\right) \frac{h_j^{(q)}h_j^{(r)}}{h_j^{(\xi)}}\le 0,
    $$
    where the boundary term $\operatorname{\mathbb{BT}}_j$ is given by \eqref{eq:boundary-term-pointwise-linear}.
\end{theorem} 
Next, we will derive the SATs for the nonlinear RSWE and prove numerical stability.
\subsubsection{SATs for the nonlinear RSWE}
As before, at each point on the boundary we set the $\operatorname{SAT}$s
\begin{equation}\label{eq:sat-non-linear}
\mathrm{SAT}^{n} =
\begin{bmatrix}
    \tau_h\left(\alpha G_n - \beta F_n- d_1\right)\\
     \tau_n\vb{n}\left(\alpha G_n - \beta F_n- d_1\right)
\end{bmatrix},
\quad
\mathrm{SAT}^{s} = 
\begin{bmatrix}
    0\\
     \tau_s\vb{m}u_n\left(u_s- d_2\right)
\end{bmatrix},
\end{equation}
where $\tau_h$, $\tau_n$ and $\tau_s$ are penalty parameters that will be determined by ensuring numerical stability. Here $\bm{n} = [n_x, n_y]^\top$  and  $\bm{m} = [m_x, m_y]^\top = [n_y, -n_x]^\top$ are, respectively, the unit normal vector on the boundary and the  unit tangential vector on the boundary. Note that $u_n =\bm{n}\cdot \bm{u} $ is the normal velocity component and  $u_s = \bm{m}\cdot \bm{u}$ is the tangential component of the velocity field on the boundary. It is also noteworthy that $\mathrm{SAT}^{n}$ and $\mathrm{SAT}^{s}$ encode the nonlinear boundary operators defined by \eqref{eq:BC-sub-critical_nonlinear_inflow},\eqref{eq:BC-sub-critical_nonlinear_outflow}, and would vanish if $G_n, F_n, u_s$ satisfy the nonlinear BCs exactly.

\noindent
As before, the following identities 
$$
\begin{bmatrix}
    G\\
    \bm{F}
\end{bmatrix}^\top \mathrm{SAT}^{n} 
= 
\tau_h G \left(\alpha G_n - \beta F_n- d_1\right) + \tau_n F_n \left(\alpha G_n - \beta F_n- d_1\right),
     \quad
\begin{bmatrix}
    G\\
    \bm{F}
\end{bmatrix}^\top\mathrm{SAT}^{s} 
= 
\tau_s  u_n u_s \left(u_s- d_1\right), 
$$
will be useful for the nonlinear stability analysis.

The $\xi$-axis SATs for the nonlinear RSWE is given by
\begin{eqnarray}
\mathrm{SAT}_\xi=
 \begin{cases}
    \bm{H}_{\xi}^{-1}\left(\mathbf{e}_{\xi, 1}\sqrt{\bm{\xi}_x^2 + \bm{\xi}_y^2}\left(\mathrm{SAT}^{n} + \operatorname{SAT}^{s}\right) + \mathbf{e}_{\xi, n_\xi}\sqrt{\bm{\xi}_x^2 + \bm{\xi}_y^2}\left(\mathrm{SAT}^{n} + \mathrm{SAT}^{s}\right)\right), \quad \ \text{if} \ u_n < 0, \\
   \bm{H}_{\xi}^{-1}\left(\mathbf{e}_{\xi,1}\sqrt{\bm{\xi}_x^2 + \bm{\xi}_y^2}\mathrm{SAT}^{n} + \mathbf{e}_{\xi,n_\xi}\sqrt{\bm{\xi}_x^2 + \bm{\xi}_y^2}\mathrm{SAT}^{n}\right), \quad \ \text{if} \ u_n > 0.
    \end{cases}
\end{eqnarray}
where $\mathrm{SAT}^{n}$ and $\mathrm{SAT}^{s}$ are given by \eqref{eq:sat-non-linear}.

Here $\mathbf{e}_{\xi,1}$ and  $\mathbf{e}_{\xi,n_\xi}$ defined in \eqref{eq:boundary-projection} are projection matrices that project the SATs to the respective boundaries $\xi_1 =0$ and $\xi_{n_\xi} = 1$. Again, we assume homogeneous boundary data $\bm{d}=0$ and introduce the numerical boundary term at every point on the boundary
\begin{equation}\label{eq:boundary-term-pointwise-nonlinear}
    \operatorname{\mathbb{BT}}_j =
\begin{cases}
-\left(G_{j}F_{nj}\right)+ \tau_h G_{j} \left(\alpha G_{nj} - \beta F_{nj}\right) +
     \tau_n F_{nj}\left(\alpha G_{nj} - \beta F_{nj}\right) + \tau_s h_{j} u_{nj} u_{sj}^2, \quad \ \text{if} \ u_{nj} < 0\\
-\left(G_{j}F_{nj}\right)+ \tau_h G_{j} \left(\alpha G_{nj} - \beta F_{nj}\right) +
     \tau_n F_{nj}\left(\alpha G_{nj} - \beta F_{nj}\right) , \quad \ \text{if} \ u_{nj} \ge  0.
     \end{cases}.
\end{equation}
Note that 
$
G_n = gh + \frac{1}{2}u_n^2 
$
and 
$G = gh + \frac{1}{2}u_n^2 + \frac{1}{2}u_s^2  = G_n + \frac{1}{2}u_s^2 $.

The following Lemma constrains the penalty parameters and ensures that the boundary term $\operatorname{BT}_j$  is never positive.
\begin{lemma}\label{lem:boundary-term-pointwise-nonlinear}
    Consider the boundary term  $\operatorname{\mathbb{BT}}_j$ defined by \eqref{eq:boundary-term-pointwise-nonlinear}, where $u_n + c >0$, $u_n-c < 0$. and assume $\alpha\ge 0$,  $\beta \ge 0$ with $|\alpha| + |\beta| >0$.
    If the boundary parameters $\alpha$, $\beta$ and the penalty weights $\tau_h$, $\tau_n$ and $\tau_s$ satisfy
    \begin{itemize}
      \item[ 1:]  When $\alpha = 0$, $\beta > 0$:  $\tau_h = -1/\beta$, $\tau_n \ge 0$ and $\tau_s \ge 1/2$.
       \item[2:]  When  $\alpha > 0$, $\beta = 0$:  $\tau_h \le 0$, $\tau_n = 1/\alpha$ and $\tau_s \ge 1/2$,
       \item[3:]  When $\alpha > 0$, $\beta > 0$:  $\tau_h =-1/(2\beta)$, $\tau_n = 1/(2\alpha)$ and $\tau_s \ge 1/2$,
    \end{itemize}
  then the boundary is never positive, that is $\operatorname{\mathbb{BT}}_j \le 0$.
\end{lemma}
\begin{proof}
Case 1:
    Consider the boundary term  $\operatorname{\mathbb{BT}}_j$ defined by \eqref{eq:boundary-term-pointwise-nonlinear} with $\alpha= 0$ and $\beta >0 $, and set $\tau_h = -1/\beta$
     we have
     $$
      \operatorname{\mathbb{BT}}_j =
\begin{cases}
     -\tau_n \beta F_{nj}^2 + \tau_s h_j u_{nj} u_{sj}^2, \quad \ \text{if} \ u_{nj} < 0\\
     -\tau_n \beta F_{nj}^2  , \quad \ \text{if} \ u_{nj} \ge  0.
     \end{cases}.
     $$
     Thus if $\tau_n \ge 0$ and $\tau_s \ge 1/2$ then $\operatorname{\mathbb{BT}}_j\le 0$.
     

\noindent
Case 2:
    Consider the boundary term  $\operatorname{\mathbb{BT}}_j$ defined by \eqref{eq:boundary-term-pointwise-nonlinear} with $\alpha> 0$ and $\beta =0 $, and set $\tau_n = 1/\alpha$
     we have
     $$
      \operatorname{\mathbb{BT}}_j =
\begin{cases}
     \tau_h \alpha G_{j}  G_{nj}  + 
    \left(\tau_s -\frac{1}{2}\right)h_ju_{nj}u_{sj}^2 , \quad \ \text{if} \ u_{nj} < 0\\
 \tau_h \alpha G_{j}  G_{nj}  
     -\frac{1}{2}h_ju_{nj}u_{sj}^2, \quad \ \text{if} \ u_{nj} \ge  0.
     \end{cases}.
     $$
where we have used
$
G_n = gh + \frac{1}{2}hu_n^2 
$
and 
$G = gh + \frac{1}{2}u_n^2 + \frac{1}{2}u_s^2  = G_n + \frac{1}{2}u_s^2 $.
Note that $G_{j} >0$,  $G_{nj} >0$, and $G_{j}  G_{nj} >0 $, so if $\alpha >0$, $\tau_h \le 0$ and $\tau_s \ge 1/2$ then $\operatorname{\mathbb{BT}}_j \le 0$.
\newline
\newline
\noindent
Case 3:
    Consider the boundary term  $\operatorname{\mathbb{BT}}_j$ defined by \eqref{eq:boundary-term-pointwise-nonlinear} with $\alpha> 0$ and $\beta >0 $, and set $\tau_h = -1/(2\beta)$ $\tau_n = 1/(2\alpha)$
     we have
     $$
      \operatorname{\mathbb{BT}}_j =
\begin{cases}
      -\frac{\alpha}{2\beta} G_{j}  G_{nj}  - \frac{\beta}{2\alpha} F_{nj}^2 + 
    \left(\tau_s -\frac{1}{4}\right)h_ju_{nj}u_{sj}^2 , \quad \ \text{if} \ u_{nj} < 0\\
-\frac{\alpha}{2\beta} G_{j}  G_{nj}  - \frac{\beta}{2\alpha} F_{nj}^2  
     -\frac{1}{4}h_ju_{nj}u_{sj}^2, \quad \ \text{if} \ u_{nj} \ge  0.
     \end{cases}.
     $$
As above, we have used
$
G_n = gh + \frac{1}{2}u_n^2 
$
and 
$G = gh + \frac{1}{2}u_n^2 + \frac{1}{2}u_s^2  = G_n + \frac{1}{2}u_s^2 $.
Note that $G_{j} >0$,  $G_{nj} >0$, and $G_{j}  G_{nj} >0 $, thus if  $\tau_s \ge 1/2 > 1/4$ then $\operatorname{\mathbb{BT}}_j \le 0$.
\end{proof}
The following theorem ensures the stability of nonlinear semi-discrete approximation \eqref{eqn:nlSWE2D-semi-discrete-SAT} for the nonlinear RSWE IBVP.
\begin{theorem}
    Consider the semi-discrete approximation  \eqref{eqn:nlSWE2D-semi-discrete-SAT} of the nonlinear vector invariant  RSWE \eqref{eqn:nlSWE2D} 
    at sub-critical flows, $ u_n + c >0$, $u_n-c < 0$. If the boundary parameters $\alpha$, $\beta$ and the penalty parameters $\tau_h,\tau_n,\tau_s$ satisfy the conditions of Lemma \ref{lem:boundary-term-pointwise-nonlinear},  then the semi-discrete approximation  \eqref{eqn:nlSWE2D-semi-discrete-SAT} is energy/entropy stable, that is, with homogeneous boundary data $\bm{d}=0$, we have
    $$
    \frac{dE_d}{dt} = \mathbb{BT}_n:= \sum_{\xi = q, r}\sum_{j=1}^{\frac{n_qn_r}{n_\xi}} \left.\left(J_j\left(\sqrt{\xi_{xj}^2 + \xi_{yj}^2}\operatorname{\mathbb{BT}}_j\right)\right|_{\xi_{1} = 0} + \left.\left(J_j\sqrt{\xi_{xj}^2 + \xi_{yj}^2}\operatorname{\mathbb{BT}}_j\right)\right|_{\xi_{n_\xi} = 1}\right) \frac{h_j^{(q)}h_j^{(r)}}{h_j^{(\xi)}}\le 0,
    $$
    where the boundary term $\operatorname{\mathbb{BT}}_j$ is given by \eqref{eq:boundary-term-pointwise-nonlinear}.
\end{theorem} 

In the next section, we will present some numerical examples to verify the analysis performed in  this paper.
\section{Numerical results}\label{sec:numerical-experiments}
In this section, we present numerical experiments for both 2D linear and nonlinear RSWE IBVPs across various mesh types to validate the theoretical analysis from previous sections. We have developed a solver to implement the SBP-SAT schemes derived earlier for these problems. Specifically, we consider a Cartesian mesh and two curvilinear meshes, as illustrated in Figure \ref{fig:mms_ini_lin}. The computational domains are representative of geometries commonly used in regional-scale and limited-area climate simulations \cite{DaviesTerry2014, Baumhefner01011982, Caron2013, TermoniaDeckmynHamdi2009, williamson1992standard}. 
The experiments are designed to verify the accuracy of the method and to numerically assess the stability and robustness of both the BCs and the SBP-SAT scheme. It is important to note that the SBP operator employed in this study is fourth-order accurate in the interior, with a second-order accurate boundary closure. For sufficiently smooth solutions, we expect a global convergence rate of third order \cite{gustafsson1975convergence, gustafsson1981convergence}. 
For time integration, we utilize the low-storage, fourth-order and five-stage explicit Runge-Kutta method \cite{carpenter1994fourth}, with the explicit time step $dt>0$ chosen accordingly for stability considerations,

\begin{align}
    dt = \frac{\mathrm{ CFL}}{c} \min_{\xi \in \{q, r\} }\sqrt{  \frac{h_{\xi}^2}{\xi_x^2+\xi_y^2}}.
\end{align}
Here $\mathrm{CFL} =0.5$ is the Courant-Friedrichs-Lewy (CFL) number, $c = |\vb{U}| + \sqrt{gH} >0$ is the background wave speed, $\xi_x, \xi_y$ are the metric derivatives of the curvilinear transformation, and  $h_{\xi}>0$ is the uniform spatial step used to discretize the $\xi$-axis in the reference domain, $(q, r)\in [0,1]^2$.
Some of the parameters used in numerical experiments that are common to both the linear and nonlinear RSWEs are summarized in Table \ref{tab:para}.
\begin{table}[h!]
\begin{center}
\begin{tabular}{|| m{2em}| m{2em} | m{2em}| m{2em}| m{4em}| m{4em} ||} 
  \hline
    $g$ & $f_c$   & $H$ & $L$ & $U$ & $V$\\
  \hline
    $9.81$ & $2$   & 2 & $100$ & $-0.5\sqrt{gH}$ & $-0.5\sqrt{gH}$ \\
  \hline
\end{tabular}
\caption{The  parameters used in the numerical experiments.}
\label{tab:para}
\end{center}
\end{table}
In the coming numerical experiments we will consider the Riemann BCs \eqref{eq:BC-sub-critical_Riemann_inflow}--\eqref{eq:BC-sub-critical_Riemann_outflow} and \eqref{eq:BC-Riemann_nonlinear_inflow}--\eqref{eq:BC-Riemann_nonlinear_outflow}, and supply boundary data accordingly. We have also verified the stability of the  mass flux BCs \eqref{eq:BC-sub-critical_massflux_inflow}, \eqref{eq:BC-sub-critical_massflux_nonlinear-inflow} and the Bernoulli's BCs \eqref{eq:BC-sub-critical_velocity-flux_inflow}, \eqref{eq:BC-sub-critical_velocity-flux_nonlinear-inflow}, at the inflow, but these are not reported here.
\subsection{Accuracy}
In this section, we verify the  accuracy of the numerical method using the method of manufactured solutions (MMS) \cite{roache_code_2001}. We consider the exact manufactured solution given by:
\begin{equation}\label{eqn:mms_lin}
\begin{aligned}
    & u_e = \cos(\pi t)\sin(5 \pi x/L)\sin(5 \pi y/L),\\
    & v_e = \sin(\pi t)\cos(5 \pi x/L)\cos(5 \pi y/L),\\
    & h_e = 1 + 0.2\cos(\pi t)\cos(5 \pi x/L)\cos(5 \pi y/L).
\end{aligned} 
\end{equation}
Figure \ref{fig:mms_ini_lin} shows the initial water height plotted on the three different geometries, a Cartesian mesh, a seashell geometry and a panel of the cube-sphere mesh. We generate, source terms, initial and boundary data to match the analytical solution \eqref{eqn:mms_lin}. We run the simulations until the final time $t =10$, on an increasing sequence of mesh resolutions with $n_\xi = 21,41,81,161,321$, $\xi \in \{q, r\}$.
\begin{figure}[tbh]
    \centering
    \includegraphics[width=1\linewidth]{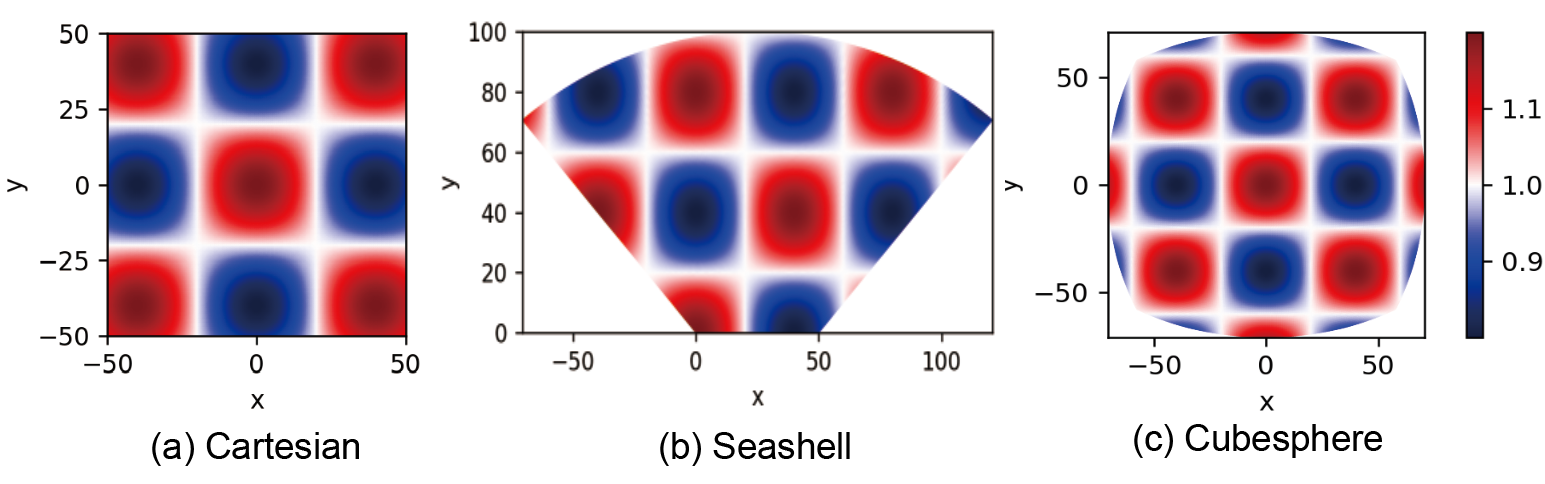}
    \caption{Plots of the water height $h$ for the MMS solution \ref{eqn:mms_lin} at $t = 0$ for different mesh types}
    \label{fig:mms_ini_lin}
\end{figure}
The numerical errors are computed by comparing  numerical solutions with the exact solution in the $l_2$ norm at the final time $t =10$. We have plotted the $l_2$ error at the final time, as a function of the grid resolution $\mathrm{dx}=1/(n_\xi-1)$. Figure \ref{fig:mms_con_lin} shows the errors and the convergence rate of the numerical errors for the linear RSWE IBVP. Similarly, Figure \ref{fig:mms_con_non} shows the numerical errors and the convergence rate of the errors for the nonlinear RSWE IBVP. Note that for linear RSWE IBVP and nonlinear RSWE IBVP the errors converge at the rate $O(\mathrm{dx}^3)$  to zero, see Figures \ref{fig:mms_con_lin}--\ref{fig:mms_con_non}. This is the expected optimal convergence rate for the numerical method \cite{gustafsson1981convergence}. We have also run the simulations for much longer times and did not observe any instability. 
\begin{figure}[thb]
    \centering
    \includegraphics[width=1\linewidth]{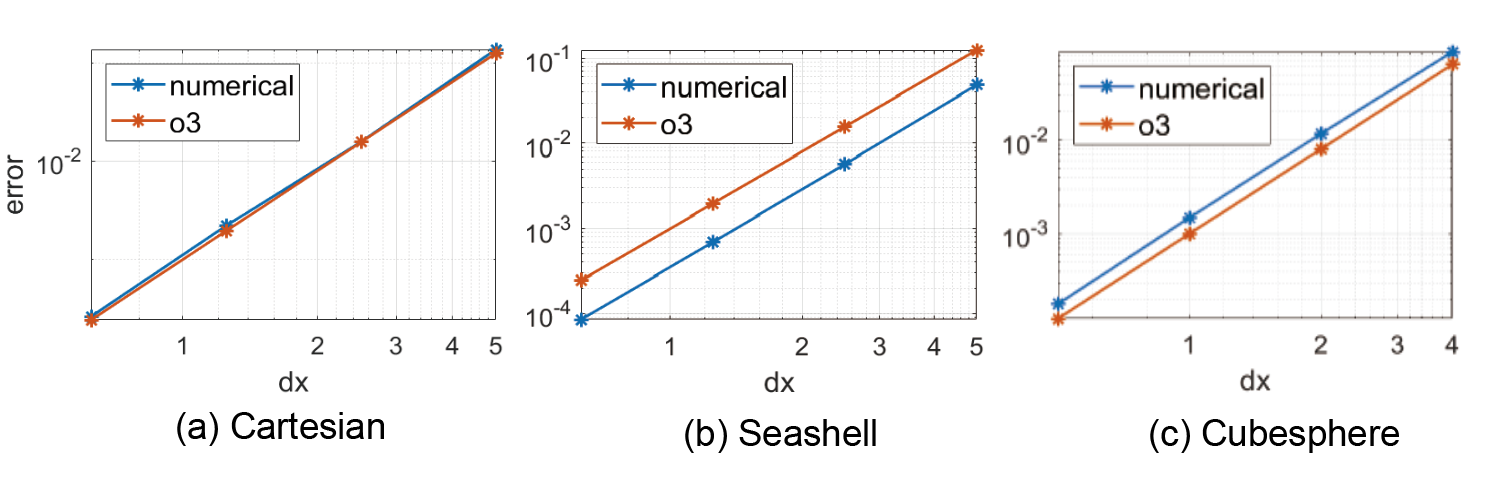}
    \caption{Convergence rate for 2D linear SWE with MMS for different mesh types}
    \label{fig:mms_con_lin}
\end{figure}
\begin{figure}[tbh]
    \centering
    \includegraphics[width=0.9\linewidth]{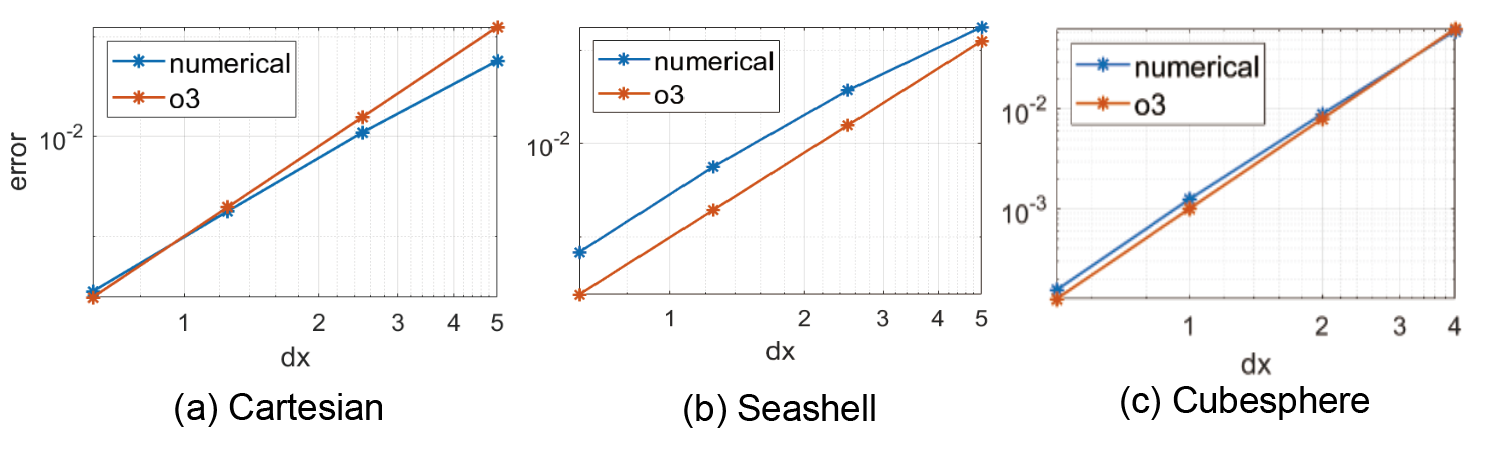}
    \caption{ Convergence rate for 2D nonlinear SWE with MMS for different mesh types}
    \label{fig:mms_con_non}
\end{figure}

\subsection{Initial Gaussian profile}
Next we will simulate   the evolution of an initial 2D Gaussian profile on the three different computational geometries and investigate numerically how the waves interact with the boundary. We consider a medium with the constant background states $U, V, H$ given in  Table \ref{tab:para}. We add the Gaussian perturbation to the initial water height only having
$h_{0}(x,y) = H + \sigma_0 \widetilde{h}_0(x,y)$ with $\sigma_0 = 0.1H$ and $\widetilde{h}_0(x,y)$ is given by 
%

\begin{equation}\label{equ:ini_con1}
    \widetilde{h}_0(x,y) =  \exp\left(-\frac{(x - x_0)^2 + (y - y_0)^2}{9}\right)
\end{equation}
 where $(x_0,y_0)$, given in Table \ref{tab:x0y0}, is the central position of the Gaussian. 
%
\begin{table}[htb]
\centering

\begin{tabular}{|c|c|c|c|}
\hline
& Cartesian & Seashell & Cubesphere \\
\hline
$(x_0,y_0)$ & $(25, 25)$ & $(50, 60)$ & $(25, 25)$ \\
\hline
\end{tabular}
\caption{Values of $(x_0,y_0)$ for different mesh types}
\label{tab:x0y0}
\end{table}

\subsubsection{Linear RSWE IBVP}
 Note that for the linear RSWE \eqref{eqn:LinSWE2D} we evolve the time-dependent perturbations since the background states $H,U,V$ are included as parameters in the linearized equations. Thus, for the  linear IBVP RSWE we set the Riemann BCs \eqref{eq:BC-sub-critical_Riemann_inflow}--\eqref{eq:BC-sub-critical_Riemann_outflow} with homogeneous boundary data. We consider  the initial condition $h_0 = \widetilde{h}_0$ for the water height and zero initial conditions for the the velocity field, that  is $u_0=0$, $v_0 = 0$.

  We discretize the reference domain $(q,r)\in [0,1]^2$ uniformly with the $n_\xi = 151$ number of grid in both directions, $\xi \in \{q, r\}$. We compute the numerical solutions until the final time $T=20$. The snapshots of the $y$-component of the particle velocity $\vb{v}$ are plotted in Figures \ref{fig:lin_car_v}--\ref{fig:lin_sea_v} for the three geometries. Figures \ref{fig:lin_car_v}--\ref{fig:lin_sea_v} show the evolution of the initial Gaussian profile, in particular how the solutions rotate, spread and transported by the constant background flow velocity field. Note that because of the non-reflecting properties of the Riemann BCs \eqref{eq:BC-sub-critical_Riemann_inflow}--\eqref{eq:BC-sub-critical_Riemann_outflow} the solutions exit the computational domain with little reflections. 
  However, the numerical reflections can be considered to be small when compared with the initial amplitude of the wave, arriving at the boundaries.
\begin{figure}
    \centering{
     %
    \includegraphics[scale =\figurescaling]{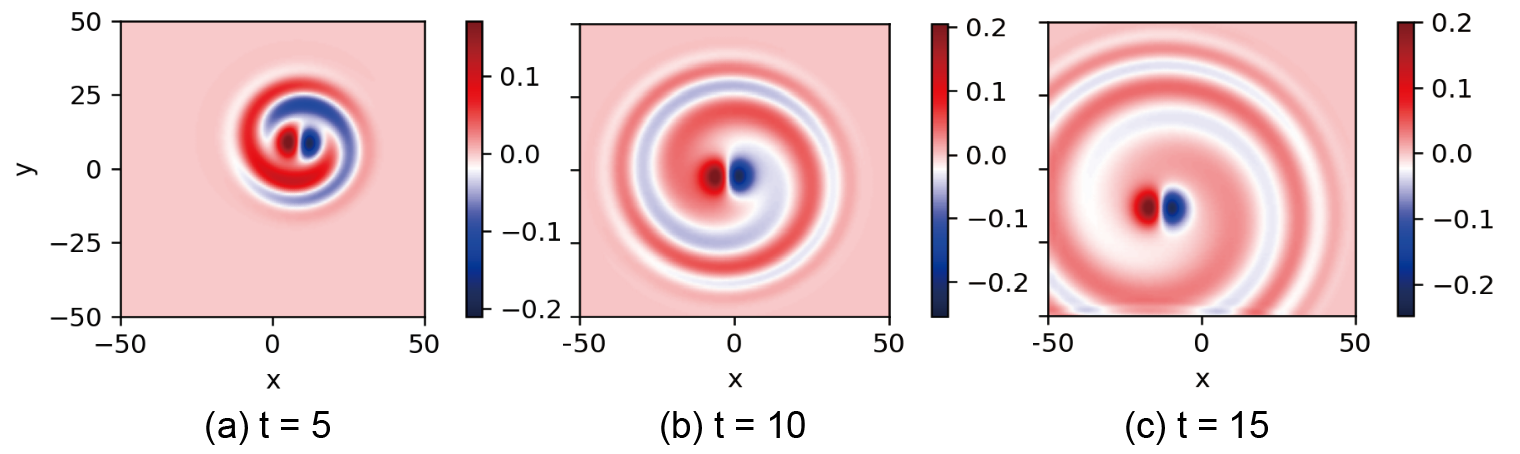}
    \caption{The snapshots of the $y$-component of the particle velocity $\vb{v}$ for  the linear RSWE on 2D rectangular geometry at $t = 5,10,15.$}
    \label{fig:lin_car_v}
    }
\end{figure}
\begin{figure}[tbh!]
    \centering
     \includegraphics[scale =\figurescaling]{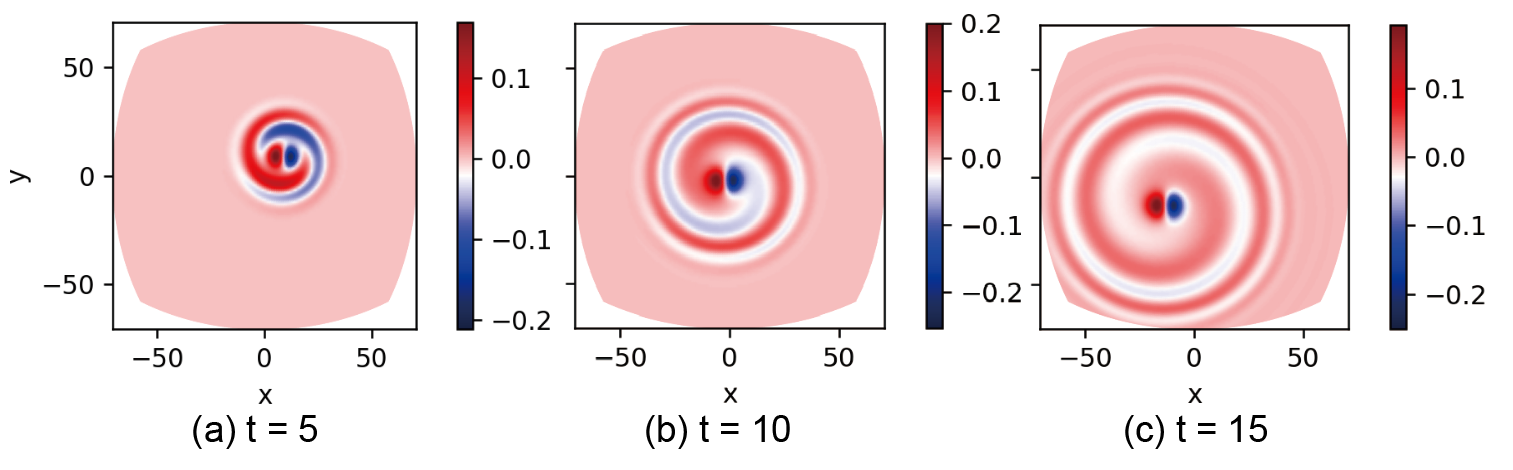}
     \caption{The snapshots of the $y$-component of the particle velocity $\vb{v}$ for  the linear RSWE on the panel of a cubesphere geometry  at $t = 5,10,15.$}
    \label{fig:lin_cub_v}
\end{figure}
\begin{figure}[tbh!]
    \centering
     \includegraphics[scale =0.585]{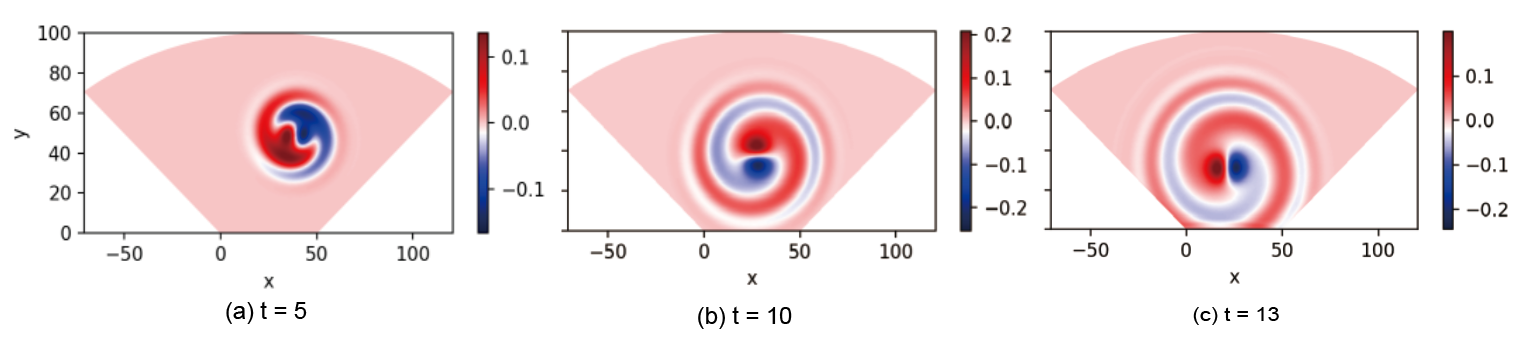}
    \caption{The snapshots of the $y$-component of the particle velocity $\vb{v}$ for  the linear RSWE on the panel of a Seashell geometry  at $t = 5,10,15.$}
    \label{fig:lin_sea_v}
\end{figure}

\subsubsection{Nonlinear RSWE IBVP}
Next, we consider the nonlinear  RSWE \eqref{eqn:nlSWE2D} with the nonlinear Riemann's BCs \eqref{eq:BC-Riemann_nonlinear_inflow}--\eqref{eq:BC-Riemann_nonlinear_outflow}. For the  nonlinear IBVP RSWE we consider   the initial condition $h_{0}(x,y) = H + \sigma_0 \widetilde{h}_0(x,y)$ for the water height and constant initial conditions for the the velocity field, that is  $u_0=U$, $v_0 = V$, with zero perturbations.  The inhomogeneous boundary data  are constructed to account for the effects of the Coriolis force term $f_c$ and ensure compatibility with the initial data, as discussed in section \ref{sec:EffectsofCoriolis}, see \eqref{eq:BC-Riemann_nonlinear_data}--\eqref{eq:BC-sub-critical_velocity-flux_nonlinear-data}. As noted earlier, this construction aims to avoid mismatch with remote boundary data.

As above, we discretize the reference domain $(q,r)\in [0,1]^2$ uniformly with the $n_\xi = 151$ number of grid in both directions, $\xi \in \{q, r\}$. We compute the numerical solutions until the final time $T=20$. The snapshots of the $y$-component of the particle velocity $\vb{v}$ are plotted in Figures \ref{fig:non_car_v}--\ref{fig:non_sea_v} for the three geometries. Figures \ref{fig:non_car_v}--\ref{fig:non_sea_v} show the nonlinear dynamics  of the initial Gaussian profile, in particular how the solutions rotate, spread and transported. The nonlinear dynamics shown in Figures \ref{fig:non_car_v}--\ref{fig:non_sea_v} are somewhat different from the linear dynamics shown in Figures \ref{fig:lin_car_v}--\ref{fig:lin_sea_v}. For the nonlinear RSWE the rotational effects of the Coriolis force term are much more pronounced than in the linear case.  We also note that because of the non-reflecting properties of the nonlinear Riemann BCs \eqref{eq:BC-Riemann_nonlinear_inflow}--\eqref{eq:BC-Riemann_nonlinear_outflow} the solutions exit the computational domain with little reflections. 
However, the numerical reflections can be considered to be small when compared with the initial amplitude of the wave, arriving at the boundaries.
\begin{figure}[tbh!]
    \centering
    \includegraphics[scale =\figurescaling]{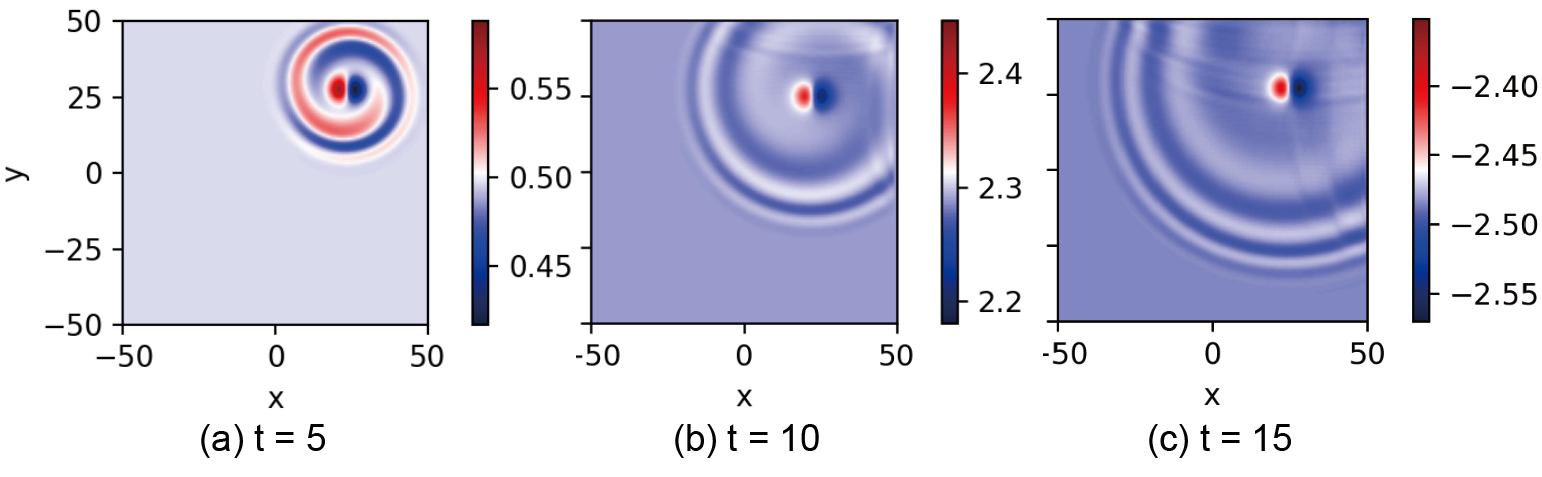}
    \caption{The snapshots of the $y$-component of the particle velocity $\vb{v}$ for  the nonlinear RSWE on 2D rectangular geometry at $t = 5,10,15.$}
    \label{fig:non_car_v}
\end{figure}
\begin{figure}[tbh!]
    \centering
     \includegraphics[scale =\figurescaling]{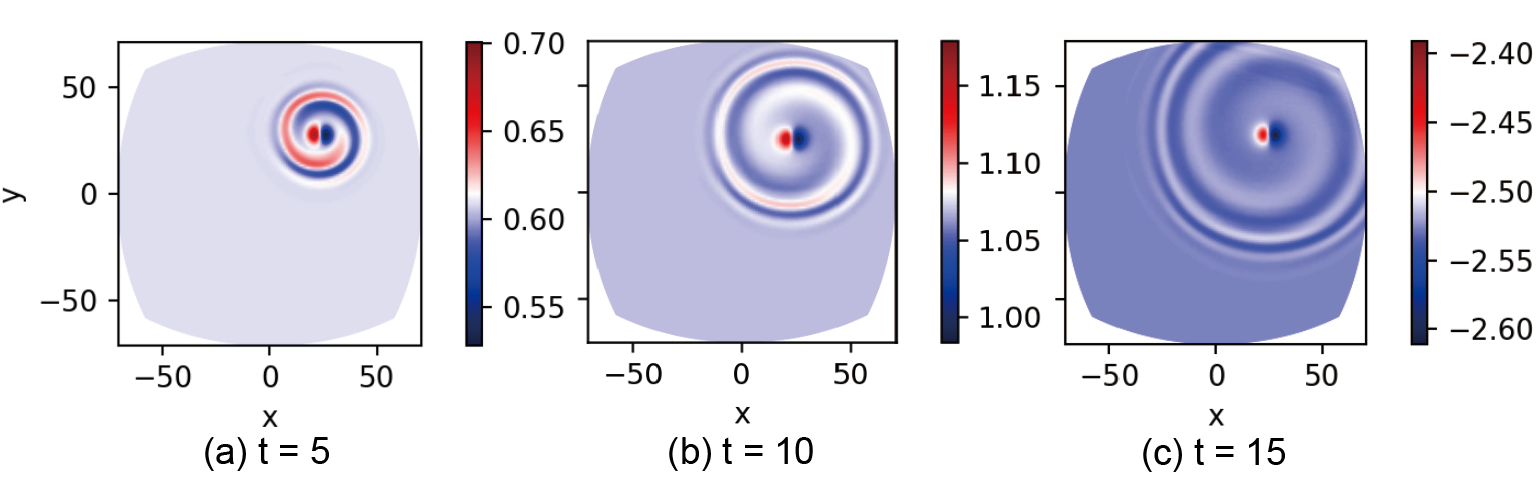}
    \caption{The snapshots of the $y$-component of the particle velocity $\vb{v}$ for  the nonlinear RSWE on the panel of a cubesphere geometry  at $t = 5,10,15.$}
    \label{fig:non_cub_v}
\end{figure}
\begin{figure}[tbh!]
    \centering
     \includegraphics[scale =\figurescaling]{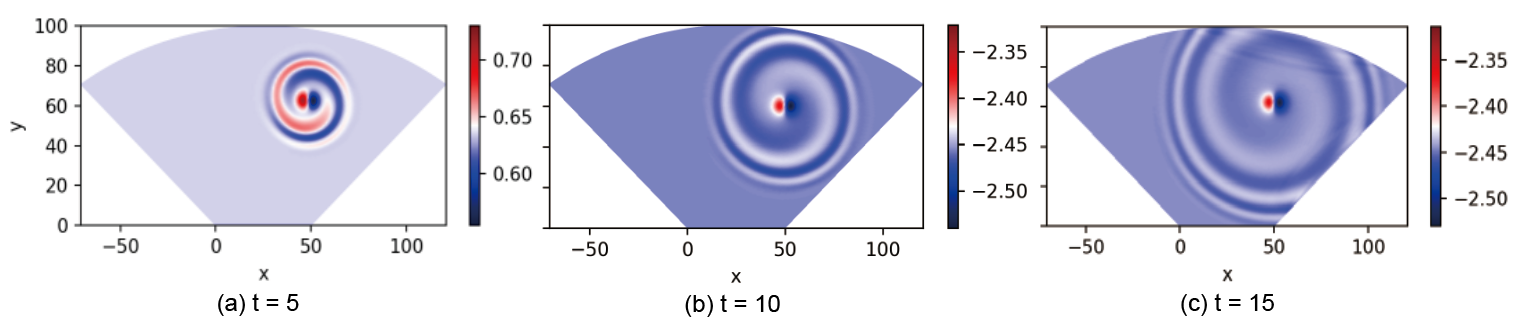}
    \caption{The snapshots of the $y$-component of the particle velocity $\vb{v}$ for  the nonlinear RSWE on the panel of a Seashell geometry  at $t = 5,10,15.$}
    \label{fig:non_sea_v}
\end{figure}

\section{Conclusions and future work}\label{sec:conclusion}
In this study, we derived and analyzed well-posed, energy- and entropy-stable BCs for the 2D linear and nonlinear RSWEs in vector invariant form on spatial domains with smooth boundaries. Our focus is on subcritical flows, which are commonly observed in atmospheric, oceanic, and geostrophic flow problems. We formulated both linear and nonlinear BCs using mass flux, Riemann's invariants, and Bernoulli’s potential, ensuring that the IBVPs are provably entropy- and energy-stable.

The linear theory developed is comprehensive and analogous to the frameworks established in \cite{GHADER20141, nordstrom2022linear, gustafsson1995time}. It provides sufficient conditions for establishing the existence, uniqueness, and energy stability of solutions to the linear IBVP for the RSWE. The nonlinear RSWE IBVP admits more general solutions, and our goal was to derive nonlinear BCs that guarantee entropy stability. To this end, we introduced the notions of linear consistency and linear stability for nonlinear IBVPs. We demonstrate that if a nonlinear IBVP is both linearly consistent and linearly stable, then, given sufficiently regular initial and boundary data over a finite time interval, there exists a unique smooth solution.

A key contribution of this work is the formulation of well-posed linear and nonlinear BCs for the RSWE in vector invariant form, tailored for high-order numerical methods. We developed high-order accurate, energy- and entropy-stable SBP-SAT numerical schemes for the corresponding linear and nonlinear IBVPs on curvilinear meshes. Detailed numerical experiments verify the accuracy of these methods and demonstrate the robustness of both the BCs and the numerical schemes. This work extends the results of \cite{HEW2025113624}, which addressed the 1D shallow water equations, to the 2D linear and nonlinear RSWE on geometrically complex domains and curvilinear meshes.

Future research will focus on extending these linear and nonlinear BCs, as well as the numerical methods, to the thermal RSWE \cite{RICARDO2024112605, RicardoDuruLee2024}, incorporating thermal effects. Additionally, we aim to adapt these approaches to the Euler equations \cite{RicardoDuruLee2024, ricardo2024thermodynamicconsistencystructurepreservationsummation}, which model compressible atmospheric flows. These advancements will significantly enhance the robustness, efficiency, and accuracy of numerical simulations in limited-area and regional atmospheric models \cite{DaviesTerry2014, Baumhefner01011982, Caron2013, TermoniaDeckmynHamdi2009, williamson1992standard} and oceanic flow models \cite{zeitlin2018geophysical} within non-periodic, bounded domains.

\bibliographystyle{unsrt} 
\bibliography{main}
\end{document}